\documentclass[11pt,a4paper]{article}
\usepackage[T1]{fontenc}
\usepackage[utf8]{inputenc}
\usepackage[english]{babel}
\usepackage{lmodern,microtype,amsmath,amssymb,amsthm,mathtools}
\usepackage[margin=25mm]{geometry}
\usepackage{enumitem,graphicx,placeins}
\graphicspath{{figures/}}
\usepackage[colorlinks=true,linkcolor=blue,citecolor=blue,urlcolor=blue]{hyperref}
\newtheorem{thm}{Theorem}[section]
\newtheorem{prop}[thm]{Proposition}
\newtheorem{lemma}[thm]{Lemma}
\newtheorem{cor}[thm]{Corollary}
\theoremstyle{definition}

\newtheorem{example}[thm]{Example}
\theoremstyle{remark}
\newtheorem{remark}[thm]{Remark}
\newcommand{\R}{\mathbb R}

\DeclareMathOperator{\rank}{rank}
\DeclareMathOperator{\diag}{diag}
\DeclareMathOperator{\sgn}{sgn}
\DeclareMathOperator{\dist}{dist}

\DeclareMathOperator{\off}{off}
\title{Finding solutions with the Borsuk--Ulam theorem}
\author{Oleg R. Musin}
\date{}
\hypersetup{pdftitle={Finding solutions with the Borsuk-Ulam theorem},pdfauthor={Oleg R. Musin}}
\AtBeginDocument{%
\setlength{\abovedisplayskip}{8pt plus 2pt minus 3pt}%
\setlength{\belowdisplayskip}{8pt plus 2pt minus 3pt}%
\setlength{\abovedisplayshortskip}{0pt plus 2pt}%
\setlength{\belowdisplayshortskip}{5pt plus 2pt minus 2pt}}
\begin{document}
\maketitle
\begin{abstract}
Many problems in nonlinear analysis, game theory, economics, and geometry
have existence proofs based on Brouwer's fixed point theorem or the
Borsuk--Ulam theorem. We study a common numerical approach: deform a
problem with one known solution into the desired problem and follow
its zeros. A constructive reduction of fixed points to odd spherical
zeros makes the approach available to a broad class of problems,
including finite-dimensional Galerkin systems, periodic solutions,
Nash and economic equilibria, and geometric mass partitions, subject
to the regularity and approximation requirements of each formulation.
On the sphere, a linear projection supplies an explicit starting model
with one antipodal pair of zeros. We give conditions for tracking the
resulting curve, allowing turns in the deformation parameter, and
relate computational work to accuracy and conditioning along the path.
A family of linear models parametrized by antipodal pairs reaches every
regular spherical solution from an open set of starts. This simple
parameter space provides a basis for exploring multiple solutions;
a finite search for all solutions needs additional coverage information.
For mass partitions of finite point sets, each continuation step has
cost proportional to the number of points in fixed dimension.
\end{abstract}


\section{Introduction}\label{sec:introduction}
Existence theorems in nonlinear analysis, game theory, economics, and
geometry often guarantee a solution without specifying how to compute
it. Brouwer's fixed point theorem and the Borsuk--Ulam theorem are two
common sources of such guarantees. Our question is numerical: given
the equations describing a solution, how can one start from a known
solution of a simpler problem and move to a solution of the given
problem? We also ask how changing the starting problem can reveal
different solutions.

\paragraph{Problems covered by the approach.}
The range of applications is best seen by starting with finite-dimensional
problems in analysis and ending with geometric partitions.
\begin{enumerate}[leftmargin=*,itemsep=4pt,topsep=5pt]
\item \emph{Galerkin approximations of nonlinear equations.}
Writing an approximate solution as $u_m=\sum_{j=1}^m a_j\phi_j$
turns a differential or operator equation into a finite system
$A_m(a)=0$. A standard Brouwer argument applies when
$\langle A_m(a),a\rangle>0$ on a sufficiently large sphere;
see, for example,~\cite{YangDai}. This condition also makes the
homotopy from $a$ to $A_m(a)$ zero-free on the boundary, providing a
particularly simple continuation model. Here we compute a solution of
the finite system; convergence to a solution of the original equation
requires separate uniform estimates and a limiting argument.
\item \emph{Periodic solutions of differential equations.}
A fixed point of the map taking an initial state forward by one forcing
period determines a periodic solution. If this map preserves a compact
convex set, Brouwer provides the existence argument~\cite{CidMawhin}.
Continuation searches for the fixed point and can therefore find
unstable as well as stable periodic solutions.
\item \emph{Variational inequalities and constrained equilibria.}
For a continuous map $A$ on a nonempty compact convex set $K$, the condition
$\langle A(x),y-x\rangle\ge0$ for every $y\in K$ is equivalent to
$x=\Pi_K(x-\lambda A(x))$, where $\lambda>0$ and $\Pi_K$ is nearest-point
projection. This gives an explicit Brouwer formulation, also covering
first-order optimality conditions for differentiable convex optimization.
\item \emph{Nash equilibria and economic equilibria.}
Nash's construction expresses equilibria of a finite game as fixed
points of a continuous self-map of the product of strategy
simplices~\cite{Nash}. Thus an explicit Brouwer-to-Borsuk--Ulam reduction
also supplies a spherical formulation of the equilibrium problem.
Economic models admitting continuous finite-dimensional fixed-point
formulations fit the same scheme; computing equilibrium prices is a
classical motivation for numerical fixed-point methods~\cite{Scarf}.
Models stated only through set-valued demand or best-response
correspondences require an additional reduction or approximation.
\item \emph{Geometric configurations and mass partitions.}
Borsuk--Ulam test maps directly encode Radon points, configurations
associated with sphere coverings, and common bisecting hyperplanes.
In particular, the ham-sandwich problem asks for a hyperplane bisecting
$d$ masses in $\R^d$. Orthogonal partitions lead to related equivariant
problems on orthogonal groups and Stiefel manifolds~\cite{MusinStiefel}.
Mass partitions provide concrete test cases and point-set complexity
estimates for the general method.
\end{enumerate}
These examples share a mechanism: formulate the problem as a zero or
fixed-point equation, choose a model with a known solution, and continue
that solution to the target. The smooth results below require the stated
rank and compactness hypotheses. Nonsmooth projections and equilibrium
maps need a suitable smooth approximation with error control, or a
separately justified PL formulation. An existence theorem by itself
does not supply these numerical conditions.

The basic idea is to deform the equations while following their zeros.
Let $f(x)=0$ be the target problem on a configuration space $M$, and
let $g(x)=0$ be a model problem with one known regular solution, counting
symmetry-related descriptions as the same solution. Join them by
\begin{equation}\label{eq:homotopy}
 F(x,t)=(1-t)g(x)+tf(x),\qquad 0\le t\le1.
\end{equation}
At $t=0$ we know where to start. For a small change of $t$, the solution
changes slightly as long as its spatial derivative remains invertible.
We predict this change and correct the predicted point back toward the
zero set. Repeating these steps leads along a solution curve.

Why should the curve reach $t=1$? Under the regularity hypotheses stated
below, the zeros form a compact one-dimensional manifold in the cylinder
$M\times[0,1]$. Its components are circles and arcs. The arc starting at
the known solution has another endpoint. It cannot return to $t=0$,
because there is no second initial solution; if zeros cannot escape
through a side boundary, it must end at $t=1$. Figure~\ref{fig:cylinder}
shows the argument after identifying symmetric descriptions. Notice that
the arc may turn back in $t$. Following a curve is therefore more general
than successively solving the equations at increasing values of $t$.
\begin{figure}[htbp]
\centering
\includegraphics[width=.70\linewidth]{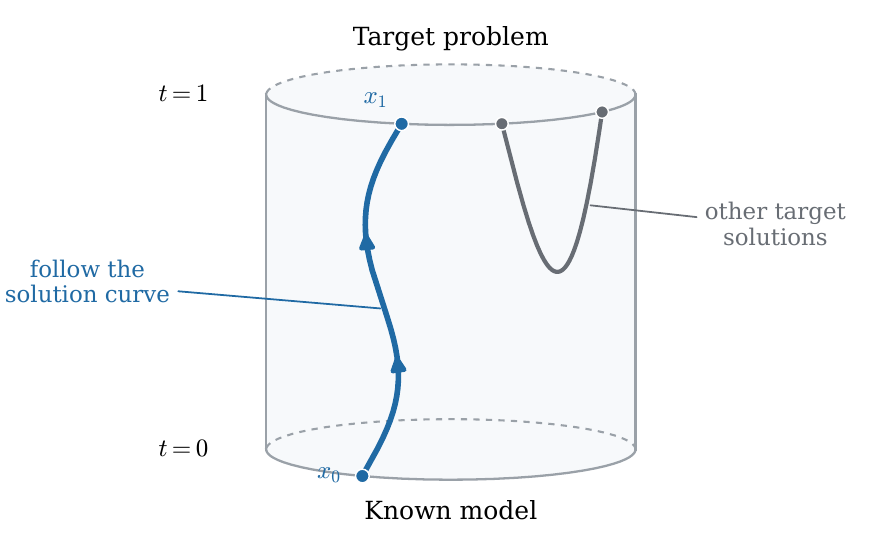}
\caption{The model has one solution at the lower boundary. Its zero arc
must reach the upper boundary, although it may turn in the height
coordinate. Other target solutions can belong to a separate arc.
Symmetry-related configurations are identified; the cylinder is schematic.}
\label{fig:cylinder}
\end{figure}

\paragraph{Why the sphere supplies a universal model.}
In its zero formulation, the Borsuk--Ulam theorem states that every
continuous odd map $f:S^n\to\R^n$ has a zero. Oddness means
$f(-x)=-f(x)$, so zeros occur in antipodal pairs. The projection
\[
 g(x_1,\ldots,x_{n+1})=(x_1,\ldots,x_n)
\]
has precisely the two zeros $\pm e_{n+1}$, one antipodal pair, and its
derivative is invertible on the tangent space at either zero. Thus every
spherical Borsuk--Ulam problem has an explicit regular model. The
homotopy~\eqref{eq:homotopy} remains odd. Its intermediate maps need not
describe physical or combinatorial instances of the original problem:
only the final zero needs to decode to the requested configuration.

\paragraph{From fixed points to spherical zeros.}
The same approach includes problems whose existence is usually proved
by Brouwer's fixed point theorem. For a self-map $T$ of a closed ball,
the equation is $x-T(x)=0$. The geometric idea of its reduction to
Borsuk--Ulam is to place this equation on an upper part of a sphere,
define the lower part by oddness, and join them through a zero-free
equatorial region. Each resulting zero pair then encodes a fixed point.
Su~\cite{Su} gives a direct construction of this kind.

For a smooth map sending the ball strictly into its interior, we use
the following particularly useful formula. Extend it to a bounded,
compactly supported smooth map $\widetilde T:\R^n\to B^n$, agreeing
with $T$ on the ball. On $S^n\subset\R^n\times\R$ set
\[
 f_T(u,s)=u-s\widetilde T(u/s)\quad(s\ne0),\qquad f_T(u,0)=u.
\]
This is odd and equals $u$ near the equator. A zero has $s\ne0$ and
gives the fixed point $x=u/s$; conversely, a fixed point gives the pair
$\pm(x,1)/\sqrt{1+\|x\|^2}$. Proposition~\ref{prop:brouwer-encoding}
proves that regularity is preserved and controls the fixed-point error
obtained from a small spherical residual. Once this encoding is made,
the same projection model is available again.

This is the constructive link between the two existence principles:
a problem first expressed through Brouwer, such as Nash equilibrium,
can be encoded as an odd spherical zero problem and then joined to a
known projection model. At the continuous level the reduction is broad;
the smooth, interior-valued case above is the setting in which our
regularity and residual estimates apply directly. The intermediate maps
need not represent games, economies, or differential equations.

\paragraph{Mass partitions as a geometric illustration.}
The ham-sandwich theorem provides a familiar illustration. Given $d$
bounded regions of positive volume in $\R^d$, one seeks a hyperplane
bisecting all of them. The regions may overlap and need not be convex
or connected. An oriented hyperplane, together with its two limiting
positions, is represented by a point of $S^d$. For each region, record
the fraction of volume on its positive side minus one half. Reversing
orientation changes every sign, giving an odd map to $\R^d$. Its zeros
are the common bisectors~\cite{StoneTukey,Matousek}.
We can therefore solve these equations starting from the spherical
projection. Alternatively, $d$ Gaussian masses with affinely independent
centers have a unique common bisector, providing a geometric model.
Orthogonal mass partitions admit a related construction on an orthogonal
group; extensions to Stiefel manifolds appear in~\cite{MusinStiefel}.

\paragraph{Finding one solution and finding all solutions.}
The selected arc reaches one target solution. Additional solutions may
be endpoints of other arcs, so following the first arc longer will not
find them. Instead, we change the model. For the sphere its known zero
is an antipodal pair, represented by a point on a hemisphere with
opposite points of its boundary identified. This gives a concrete
space of starting parameters to explore.

The simple geometry of this parameter space is a main advantage of the
spherical formulation. We search a hemisphere, not a space of arbitrary
maps: varying a point chooses a model and its known zero simultaneously.
This turns the search for all regular solutions into a concrete coverage
problem in $n$ parameters, particularly suitable for small dimensions.

The position of the pair alone does not specify the coordinates of a
projection into $\R^n$. We fix these coordinates using the derivative of
the target at the chosen start. Theorem~\ref{cor:all-sphere} proves that
every regular target zero pair is reached from an open set of these
parameters. The derivative is evaluated at the start, so constructing
a model does not require knowing the target zero in advance.
One way to explore this space is to grow a region from a seed, mark it,
and restart in the unexplored part. Several disconnected regions may
lead to the same answer. This flood-fill procedure, informally called
the rodent method, is an implementation choice, not part of the
reachability theorem. Symmetries, explicit equations, or other structure
of a particular problem may permit a more efficient search. The PL
formulation gives finite parameter decompositions on which region growing
is precise. Reachability guarantees that
each solution has successful starts; a finite stopping rule additionally
requires a suitable decomposition or quantitative coverage information.

For point sets, a hyperplane can move between points without changing
their counts. We replace the points by small smooth masses and solve
the differentiable equations. For odd color sizes in general position,
the points nearest an accurate smooth cut determine an exact discrete
bisector, which we verify by counting points on either side.

\paragraph{Relation to continuation theory and computational cost.}
The arc argument and numerical continuation are classical
\cite{AlexanderYorke,ChowEtAl,AllgowerGeorg}. Constructive Borsuk--Ulam
proofs include~\cite{Barany,MeyersonWright,Matousek}; manifold versions
occur in~\cite{Musin2012}. Numerical fixed-point methods for economic
and game-theoretic equilibria are likewise established
\cite{Scarf,HeringsPeeters2001,GovindanWilson,HeringsPeeters2010}.
The predictor--corrector scheme and the fixed-point reduction are not
claimed as new.

Our purpose is to bring these tools together with explicit test maps,
regularity-preserving encodings, a model family reaching every regular
spherical zero, and bounds that account for the branch actually followed.
Theorem~\ref{thm:numerical} bounds the error after prediction and correction;
Proposition~\ref{prop:adaptive-geometry} relates the step count to local
acceleration, conditioning, and requested accuracy. As in condition-based
polynomial continuation~\cite{ShubCondition,BeltranLeykin}, the work per
step and the number of steps must be distinguished.

For $N$ input points in fixed dimension, each step costs $O(N)$.
Uniform step-count bounds therefore give linear total cost in $N$.
There are also exact linear-time planar algorithms: for ham-sandwich
cuts~\cite{LoMatousekSteiger} and for orthogonal four-partitions of one
point set~\cite{FakhrutdinovMusin}. Their higher-dimensional extensions
have dimension-dependent polynomial complexity. Continuation offers an alternative
whose cost depends on accuracy and branch geometry, without an efficiency
guarantee from topology alone.

Sections~\ref{sec:topology}--\ref{sec:all} give the continuation principle,
numerical estimates, and all-solution search. Sections~\ref{sec:partitions}
and~\ref{sec:other} develop mass partitions and other spherical encodings;
Section~\ref{sec:experiments} presents the computations.

\section{Zero arcs and PL continuation}\label{sec:topology}
The underlying fact is not specific to a sphere or an involution.
We state the ordinary version first and then apply the same argument
after identifying symmetric solutions.

Throughout, $S^n\subset\R^{n+1}$ is the unit sphere, $B^n$ is the open
unit ball, and $\overline B^n$ is its closure. A closed manifold is
compact and has no boundary. Norms are Euclidean and their induced
operator norms unless specified otherwise; $\dist$ on a manifold is
its Riemannian distance. For a homotopy, $D_xF$ denotes the spatial
derivative and $DF=D_{(x,t)}F$ its full derivative. Regularity of a zero
means surjectivity of the relevant derivative; for a map between spaces
of the same dimension this is invertibility.

\begin{prop}[Continuation without symmetry]\label{prop:general-cylinder}
Let $M$ be a smooth $n$-manifold, possibly noncompact or with boundary,
and let $F:M\times[0,1]\to\R^n$ be $C^1$. Suppose its zero set is compact,
does not meet $\partial M\times[0,1]$, and the initial slice has exactly
one zero. Assume $D_xF$ is invertible at all endpoint zeros and $DF$ is
surjective at every interior zero. Then the component of the initial
zero is a compact regular arc with its other endpoint at $t=1$.
\end{prop}
\begin{proof}
The regular level theorem, including endpoint collars, makes the zero
set a compact one-manifold with boundary exactly its endpoint zeros.
Its component through the initial zero is an interval. Its other endpoint
cannot be another initial zero and must therefore be a target zero.
\end{proof}

Compactness and boundary exclusion prevent a path from escaping; regularity
prevents it from ending at an interior singularity. Uniqueness alone is
insufficient: $F(x,t)=x^2+t$ loses its unique, nonregular initial zero.
An ordinary map from a closed $n$-manifold to $\R^n$ cannot have just one
regular zero (its mod-two zero count is even). On the sphere the model
has one \emph{orbit}, consisting of two antipodal points.

Let $n\ge1$ and let $M$ be a closed smooth $n$-manifold, possibly disconnected, with
a free smooth action of a finite group $G$. Let $V$ be an
$n$-dimensional real representation of $G$. Metrics are chosen
$G$-invariant. A map is \emph{equivariant} if $f(ax)=af(x)$ for
$a\in G$, and is \emph{regular at its zeros} if its differential
there is invertible. Solutions are counted modulo $G$.

\begin{thm}[Continuation from a regular model]\label{thm:continuation}
Let $F:M\times[0,1]\to V$ be $C^1$ and equivariant, with
$g=F(\cdot,0)$ and $f=F(\cdot,1)$. Suppose that
$g^{-1}(0)=Gx_0$, both boundary maps are regular at their zeros,
and $\rank DF=n$ at every interior zero. The zero component through
$(x_0,0)$ is a compact regular arc ending at $t=1$.

If $J=D_xF$ is invertible along this arc, it is a graph $x=x(t)$ and
\begin{equation}\label{eq:ode}
 x'(t)=-J(x(t),t)^{-1}F_t(x(t),t),\qquad x(0)=x_0.
\end{equation}
If $J$ is invertible at every zero of $F$, every time slice has
exactly one zero orbit.
\end{thm}

\begin{proof}
The regular level theorem makes $Z=F^{-1}(0)$ a one-manifold in
the cylinder's interior. At a boundary zero, spatial invertibility
expresses $Z$ as a graph over an inward collar coordinate. Thus $Z$
is a compact $C^1$ one-manifold with
$\partial Z=g^{-1}(0)\sqcup f^{-1}(0)$. The free action restricts
to $Z$, so $Z/G$ is also a compact one-manifold.

The component containing the unique initial zero orbit is an interval.
Its two endpoints are distinct. The initial boundary contains only
one point of $Z/G$, so the other endpoint must lie on the target boundary.
The covering $Z\to Z/G$ is trivial over an interval; its lift through
$(x_0,0)$ is the asserted arc. A finite cover by regular coordinate
graphs shows that its length is finite, and the inward direction
fixes its orientation.

A tangent vector $(\xi,\tau)$ satisfies $J\xi+F_t\tau=0$.
When $J$ is invertible, a nonzero tangent has $\tau\ne0$.
Its sign is initially positive and cannot change, so projection to
$t$ is a diffeomorphism onto $[0,1]$. Differentiation gives
\eqref{eq:ode}. If $J$ is invertible at every zero, no zero component
can be a circle or an interval with both endpoints on the same
boundary: $t$ would have a critical point. Hence every quotient
component joins the two boundaries, and the unique initial endpoint
leaves exactly one component.
\end{proof}

More generally, if there are $r$ initial zero orbits, $a$ quotient
intervals joining two initial endpoints, and $c$ crossing intervals,
then $r=2a+c$. Odd $r$ guarantees a crossing, but does not identify
a crossing from every start.

\begin{lemma}[Relative regularization]\label{lem:regularize}
A smooth equivariant homotopy whose boundary maps are regular can
be perturbed arbitrarily little, keeping it fixed near the boundary,
to make its full derivative surjective at every interior zero.
\end{lemma}
\begin{proof}
Equivariant maps correspond to sections of the associated bundle
$(M\times[0,1])\times_GV$ over the quotient cylinder.
Boundary regularity gives regular collars. Choose smaller collars
on which the perturbation is identically zero. On the compact
complement of the larger collars, finitely many local frames
multiplied by bump functions, supported away from the smaller
collars, span every fiber. Adding their small linear combinations
gives a transverse universal section on this compact region;
on the collars transversality persists for small coefficients. Sard's theorem
applied to its zero set and projection to the coefficients gives
transverse sections for almost every small choice. Smallness preserves
regularity in the collars. Lifting proves the assertion.
\end{proof}

\begin{prop}[Continuous zero theorem]\label{prop:continuous}
If a smooth equivariant $g:M\to V$ has one regular zero orbit,
the zero set of every continuous equivariant homotopy starting at $g$
contains a compact connected subset containing $(x_0,0)$ and meeting $t=1$.
In particular, every continuous equivariant $f:M\to V$ has a zero.
\end{prop}
\begin{proof}
Flatten the homotopy on a shrinking initial collar, smooth in charts,
and average over $G$, keeping $g$ fixed. Perturb the target to a transverse
section and use Lemma~\ref{lem:regularize}, with all errors tending to zero.
Theorem~\ref{thm:continuation} gives crossing arcs from $(x_0,0)$.
Compactness of the cylinder gives a Hausdorff-convergent subsequence of
these compact sets. Its limit contains $(x_0,0)$, meets $t=1$, and lies
in the original zero set by uniform convergence. It is connected: a
separation into disjoint compact sets would separate all sufficiently
close arcs. Apply this to~\eqref{eq:homotopy} for the final assertion.
\end{proof}

\begin{cor}[The universal spherical model]\label{cor:sphere-model}
Every odd map $f:S^n\to\R^n$ has the regular one-orbit model
$g(x)=(x_1,\ldots,x_n)$. The homotopy $(1-t)g+tf$ is odd.
For continuous $f$ it has a compact connected zero set from a model
zero to the target. For $C^1$ maps satisfying the rank hypotheses
of Theorem~\ref{thm:continuation}, the selected component is a regular arc;
spatial invertibility makes it a graph over $t$.
\end{cor}
\begin{proof}
The zeros of $g$ are $\pm e_{n+1}$, and its differential on their
tangent spaces is the identity. Apply Theorem~\ref{thm:continuation}
and Proposition~\ref{prop:continuous}.
\end{proof}
Thus spherical applications need only supply an odd test map and
decode its zero. The model need not come from an instance of the
application. The projection is classical; the additional hypotheses
concern numerical tracking, not the existence of a starting model.

This limiting connected set need not be an arc. Numerical path
following therefore uses the rank assumptions, not just the continuous
existence statement. The same quotient-interval proof works on a
compact free $G$-bordism with the prescribed boundary maps; this is
the setting of~\cite[Theorem~1]{Musin2012}. All continuation constructions in this paper use cylinders.

\begin{example}\label{ex:selection}
For the antipodal circle,
\[
 F(\theta,t)=(1-t)\sin\theta+t\sin3\theta
 =\sin\theta(1-2t+4t\cos^2\theta).
\]
The branch $\theta=0$ has $F_\theta=1+2t>0$ and reaches one
of three target zero pairs. The others are joined by
$t=(4\sin^2\theta-2)^{-1}$, $\pi/3\le\theta\le2\pi/3$.
Its fold at $(\pi/2,1/2)$ is regular because $F_t=-2$.
Replacing $\theta$ by $\theta-2j\pi t/3$, $j=0,1,2$,
preserves the endpoint maps and selects different target pairs.
\end{example}

\begin{prop}[Finite PL continuation]\label{prop:pl}
Let $K$ be a finite equivariant triangulation of the cylinder over
a closed free $G$-PL $n$-manifold, with both boundary slices as subcomplexes.
Suppose $F$ is equivariant and affine on each simplex,
$F^{-1}(0)\cap K^{(n-1)}=\varnothing$, and its initial zeros form
one orbit. Following zero segments from that orbit reaches the
target after finitely many simplex transitions. Rational vertex
values allow exact rational transitions in barycentric coordinates.
\end{prop}
\begin{proof}
The dimension count is $(n+1)-n=1$ in a top simplex, $n-n=0$
in an $n$-face, and negative in smaller faces. Thus crossings of
$n$-faces are allowed; only the $(n-1)$-skeleton is excluded.
In a top-dimensional simplex with vertices $v_i$, the zero set is
\[
 \{\lambda_i\ge0:\ \textstyle\sum_i\lambda_i=1,
                    \ \sum_i\lambda_iF(v_i)=0\}.
\]
A vertex of this polytope has at most $n+1$ positive coordinates.
Skeleton avoidance forces exactly $n+1$, and the corresponding
columns $(1,F(v_i))$ are independent: otherwise a dependence could
move positive coordinates until one vanished, giving a forbidden
zero. The nonempty zero polytope is consequently a line segment.
Its endpoints lie in interiors of $n$-faces; invertibility of the
face matrix lets a zero extend into either adjacent top simplex.
Interior faces have two such simplices and boundary faces one.
The segments therefore form a finite one-manifold. The quotient
interval argument proves the assertion. The selected interval
cannot revisit a segment, and its endpoints are solutions of
linear systems with the given vertex values.
\end{proof}

This is the finite path construction underlying the PL proofs
in~\cite[Section~2.2]{Matousek} and~\cite[Lemmas~3--5]{Musin2012}.

\section{Small-step continuation}\label{sec:numerical}
The preceding section supplies a path. We now construct small
steps that stay near it and bound the error left after each correction.
The basic scheme is the classical Euler--Newton predictor--corrector
method~\cite[Chapter~5]{AllgowerGeorg}; the estimates below make its
branch-dependent constants explicit.

Write $J=D_xF$ and $b=F_t$. In a local chart the method is
\begin{equation}\label{eq:algorithm}
 \begin{aligned}
 J(y_k,t_k)v_k&=-b(y_k,t_k),&p_k&=y_k+hv_k,\\
 J(p_k,t_k+h)s_k&=-F(p_k,t_k+h),&y_{k+1}&=p_k+s_k.
 \end{aligned}
\end{equation}
These are tangent prediction and one Newton correction. On a manifold
the additions are chart operations. The gradient of $\|F\|^2/2$
vanishes at zeros and does not itself give a continuation direction.

\begin{thm}[Tracking and final refinement]\label{thm:numerical}
Suppose a $C^2$ homotopy $F$ has a zero graph $x(t)$,
$0\le t\le1$, with a known initial zero, and $J$ is invertible
along the graph, including its endpoints. There are constants
$h_0,\eta_0,C_0,C_1>0$ such that, for any requested maximum step
$\bar h>0$, with
\begin{equation}\label{eq:mesh}
 K=\left\lceil\max\{\bar h^{-1},h_0^{-1}\}\right\rceil,
 \qquad h=K^{-1},\qquad t_k=kh,
\end{equation}
algorithm~\eqref{eq:algorithm} is well defined and preserves the
error bound $\eta_0$. From the exact initial zero it satisfies
$\dist(y_k,x(t_k))\le C_0h^4$ for $1\le k\le K$.
Further Newton correction at $t=1$ attains residual at most $\delta$
for any $\delta>0$, in at most
\begin{equation}\label{eq:main-count}
 m=\left\lceil\log_2\left(1+\log_2^+(C_1/\delta)\right)\right\rceil,
 \qquad \log_2^+s=\max\{0,\log_2s\},
\end{equation}
iterations. The total number of linear solves is at most $2K+m$.
\end{thm}

Here and below the tracking bounds first refer to exact arithmetic.
We give explicit constants and then account for evaluation errors.
In a fixed Euclidean chart assume, on the tube
$\|y-x(t)\|\le\rho$, that
\begin{equation}\label{eq:bounds}
 \sigma_{\min}(J(x(t),t))\ge\alpha>0,\quad
 \|b(x(t),t)\|\le B,\quad \|D^2F\|\le L,\quad \|J\|\le\Lambda,
\end{equation}
where $L>0$ bounds the full Hessian. Put
\begin{equation}\label{eq:constants}
 \begin{gathered}
 V=B/\alpha,\quad \kappa=L/\alpha,\quad C=\kappa(1+V^2),
 \quad D=2\kappa(1+V),\\
 r=\min\{\rho,(4\kappa)^{-1}\},\quad \eta=r/8,\quad
 h_* =\min\{1,D^{-1},\sqrt{r/(2C)}\}.
 \end{gathered}
\end{equation}

\begin{lemma}\label{lem:estimates}
One has $\|x'\|\le V$ and $\|x''\|\le C$. For
$e=\|y-x(t)\|\le r$, $v=-J(y,t)^{-1}b(y,t)$, and
$\mathcal N_t(y)=y-J(y,t)^{-1}F(y,t)$,
\begin{align}
 \|v-x'(t)\|&\le De,\label{eq:velocity-error}\\
 \|y+hv-x(t+h)\|&\le(1+Dh)e+Ch^2/2,\label{eq:prediction-error}\\
 \|\mathcal N_t(y)-x(t)\|&\le\kappa e^2,
 \qquad \alpha e/2\le\|F(y,t)\|\le\Lambda e.
 \label{eq:newton-error}
\end{align}
\end{lemma}
\begin{proof}
Differentiate $F(x(t),t)=0$ twice:
$Jx'+b=0$ and $Jx''+D^2F[(x',1),(x',1)]=0$.
These give the first bounds. On the segment from $x(t)$ to $y$,
the Hessian bound gives changes in both $J$ and $b$ of at most $Le$.
Hence $\sigma_{\min}(J(y,t))\ge3\alpha/4$ and
$\|J(y,t)^{-1}\|\le2/\alpha$. Subtracting $J(y,t)v=-b(y,t)$ from
the exact velocity equation gives
$\|v-x'\|\le(2L/\alpha)(1+V)e=De$.
Taylor's formula gives~\eqref{eq:prediction-error}. With $x=x(t)$, the Newton error is
\[
 J(y,t)^{-1}\int_0^1[J(y,t)-J(x+s(y-x),t)](y-x)\,ds.
\]
Its norm is at most $Le^2/\alpha$. Expansion at $x$ gives the
lower residual bound $\alpha e-Le^2/2\ge\alpha e/2$;
integration of $J$ gives the upper bound.
\end{proof}

\begin{lemma}[Uniform chart margins]\label{lem:charts}
There are finitely many time intervals $I_j$ and charts $\psi_j:U_j\to\R^n$
with uniform constants $Q\ge1$, $\rho,h_{\rm ch}>0$. Every time interval
of length at most $h_{\rm ch}$ lies in an $I_j$. For $t\in I_j$, the
coordinate radius-$\rho$ ball about $\psi_j(x(t))$ and the intrinsic ball
$B_M(x(t),\rho/Q)$ lie in the chart, and
\begin{equation}\label{eq:chart-distances}
 Q^{-1}\dist(p,q)\le\|\psi_j(p)-\psi_j(q)\|\le Q\dist(p,q).
\end{equation}
The bounds~\eqref{eq:bounds} hold uniformly on these neighborhoods.
\end{lemma}
\begin{proof}
Take a finite cover by smaller charts compactly contained in larger ones.
Compactness gives uniform metric and derivative bounds and positive
margins; a Lebesgue number of the time cover supplies $h_{\rm ch}$.
\end{proof}

\begin{proof}[Proof of Theorem~\ref{thm:numerical}]
First use a fixed chart and take $h_0=h_*$, $\eta_0=\eta$.
An error $e\le r/8$ gives predictor error at most
$2e+r/4\le r/2$. Correction gives error at most
$\kappa r^2/4\le r/16$, so every step stays in the tracking tube.
For the sharper estimate put $q=Ch^2/2\le r/4$.
An error at most $q$ gives predictor error at most $3q$ and
corrected error at most $9\kappa q^2\le q$, since
$9\kappa q\le9/16$. Starting at zero error, induction yields
\begin{equation}\label{eq:fourth-order}
 e_k\le9\kappa q^2=\frac{9}{4}\kappa C^2h^4\quad(k\ge1).
\end{equation}
The error left by each preceding correction is included in this
induction. At the target, $e_{j+1}\le\kappa e_j^2$ and
$\kappa e_0\le\kappa\eta\le1/32$, whence
$e_j\le\eta\,2^{-(2^j-1)}$. Taking $C_1=\Lambda\eta$
proves~\eqref{eq:main-count}. Residual $\delta$ implies coordinate
error at most $2\delta/\alpha$.

For a manifold, use Lemma~\ref{lem:charts}. At each step choose a
chart containing its whole graph segment and use this chart for both
prediction and correction. Set
\begin{equation}\label{eq:chart-step}
 h_0=\min\{h_{\rm ch},1,D^{-1},\sqrt{r/(2CQ^2)}\},
 \qquad\eta_0=r/(8Q^3).
\end{equation}
Set $E_k=\dist(y_k,x(t_k))$. In the chart of each step,
$e_k\le QE_k$ and $E_{k+1}\le Qe_{k+1}$.
If $E_k\le\eta_0$, prediction gives coordinate error at most
$r/(2Q^2)$ and correction gives intrinsic error at most
$r/(16Q^3)\le\eta_0$. For the sharper bound put $q=Ch^2/2$.
Then $q\le r/(4Q^2)$ and
$E_k\le q/Q$ implies $E_{k+1}\le9Q\kappa q^2\le q/Q$.
Induction gives $C_0=(9/4)Q\kappa C^2$.
At the target the coordinate error is
at most $r/(8Q^2)\le\eta$, so the same final Newton bound applies.
\end{proof}

If a complete prediction--correction step costs $\mathcal W$
and an endpoint Newton iteration costs at most the same amount,
the total arithmetic work is
\begin{equation}\label{eq:work}
 O\bigl((K+m+1)\mathcal W\bigr).
\end{equation}
This separates the continuation count from the representation of the map.

\subsection{Local steps and inexact correction}
The global bounds above can be replaced by local tests as the
computation advances. We also allow errors in function evaluation
and in solving the linear systems. At $(y,t_0)$ let
$A=J(y,t_0)$, $\beta=\|A^{-1}\|$, $v=-A^{-1}b(y,t_0)$,
$d_0=\|A^{-1}F(y,t_0)\|$, and $W=\sqrt{1+\|v\|^2}$.
Suppose $\|D^2F\|\le L$ on the predictor tube
\[
 \{(y+sv+u,t_0+s):0\le s\le h,\ \|u\|\le R\},
 \qquad t_0+h\le1.
\]
If
\begin{equation}\label{eq:tube-test}
 \beta L(R+hW)\le\tfrac12,\qquad
 d_0+\tfrac12\beta Lh^2W^2\le R/2,
\end{equation}
each slice contains exactly one zero in its radius-$R$ ball,
and these zeros form a regular graph. Indeed,
$T_s(u)=u-A^{-1}F(y+sv+u,t_0+s)$ has derivative norm at most
$1/2$, while $\|T_s(0)\|\le R/2$ by Taylor expansion and
$Av+b=0$. The contraction theorem and the implicit function
theorem prove the claim. If the selected initial zero lies in
the first ball, uniqueness identifies this graph with its continuation.

Use separate error bounds $\tau_F,\tau_J,\tau_b$ for evaluating
$F,J,b$, in the fixed coordinate and target norms. Let $\tau_p$
and $\tau_c$ bound the residuals in the predictor and corrector
linear solves, respectively. If $\tau_J\le\alpha/4$, then
$\|\widetilde J^{-1}\|\le2/\alpha$ throughout the radius-$r$ tube.
Subtracting the predictor equations, with $v=-J^{-1}b$, gives
\begin{equation}\label{eq:inexact-velocity}
 \|\widetilde v-v\|\le
 \frac2\alpha(\tau_b+\tau_p+\tau_J\|v\|).
\end{equation}
For correction, put $\zeta=\tau_J$ and $\tau=\tau_F+\tau_c$.
The Newton integral identity gives, for predictor error $e\le r$,
\begin{equation}\label{eq:inexact}
 e_{\rm new}\le\kappa e^2+2\zeta e/\alpha+2\tau/\alpha.
\end{equation}
These tolerances enter with different factors; they need not be
equal. At $t=1$ the valid residual test is
$\|\widetilde F(y,1)\|+\tau_F\le\delta$.
Quadratic refinement requires decreasing evaluation errors:
$\zeta=O(e)$ and $\tau=O(e^2)$ suffice.

For $C^1$ maps the same argument uses moduli of continuity in
place of Hessian bounds. No rate depending only on the requested
accuracy follows from $C^1$ regularity alone.

\subsection{Residual accuracy and evaluation cost}
Fix the norm and normalization of the test map. An $\varepsilon$-solution
means $\|f(y)\|\le\varepsilon$; this definition does not require choosing
an exact zero in advance. Evaluation cost includes any integrations and
derivative calculations needed by the application.

\begin{prop}[General residual and work bound]\label{prop:oracle-work}
Under Theorem~\ref{thm:numerical}, exact evaluation with $h=1/K\le h_0$
gives $\|f(y_K)\|\le R_0h^4$, where $R_0=\Lambda QC_0$. Thus
\begin{equation}\label{eq:general-epsilon-work}
 K=\left\lceil\max\{h_0^{-1},(R_0/\varepsilon)^{1/4}\}\right\rceil,
 \qquad T(\varepsilon)=O(\mathcal W K)
\end{equation}
suffice, with $\mathcal W$ the cost of one complete step.
Let $\mathcal W(\tau)$ instead bound a step with all evaluation and
linear-solve errors from~\eqref{eq:inexact-velocity}--\eqref{eq:inexact}
at most $\tau$, in fixed coordinate and target norms. There are constants
$\widehat h_0,\widehat R_0>0$ depending on the tracking bounds such that
\begin{equation}\label{eq:inexact-oracle-work}
 K=\left\lceil\max\{\widehat h_0^{-1},
             (\widehat R_0/\varepsilon)^{1/4}\}\right\rceil,
 \qquad T(\varepsilon)=O\bigl(K\mathcal W(K^{-4})\bigr).
\end{equation}
For uniform tracking bounds and cost nonincreasing in $\tau$, this is
$O(\mathcal W(c\varepsilon)\varepsilon^{-1/4})$ as
$\varepsilon\to0$, for a uniform $c>0$.
\end{prop}
\begin{proof}
The exact bound follows by integrating $J$ from the final zero to the
computed point and multiplying the step count by its cost.
For inexact evaluation take $\tau=h^4$. Perturbed Jacobians remain
uniformly invertible for small $h$, and~\eqref{eq:inexact-velocity}
gives velocity error at most $C_vh^4$. Choose $P>C/2$ and then
$A>Q(\kappa P^2+4/\alpha)$. If the current intrinsic error is at most
$Ah^4$, its predictor error in the step's chart is at most
$2QAh^4+Ch^2/2+C_vh^5\le Ph^2$ for small enough $h$.
The value error and corrector solve residual together contribute at most
$2h^4$. Equation~\eqref{eq:inexact} therefore bounds the corrected
intrinsic error by
$Q(\kappa P^2+2Ph^2/\alpha+4/\alpha)h^4\le Ah^4$.
Induction proves the claim with $\widehat R_0=\Lambda QA$ and a smaller
step bound. Finally $K=O(\varepsilon^{-1/4})$ and
$K^{-4}\ge c\varepsilon$ for small $\varepsilon$.
\end{proof}

With value error at most $\tau$, the residual stopping test is
$\|\widetilde f(y)\|+\tau\le\varepsilon$. For endpoint accuracy alone,
fixed tracking steps followed by Newton refinement can be cheaper than
refining every node. The constants depend on the selected branch;
topological existence alone gives no accuracy-only rate.

\subsection{Geometry, conditioning, and adaptive work}\label{sec:geometry}
The same zero curve can be easy or difficult to follow depending on
its parametrization and on the local inverse bounds. We make these
dependencies explicit before discussing particular representations.

Substitution in~\eqref{eq:chart-step} gives
\begin{equation}\label{eq:condition-count}
 h_0^{-1}=\Gamma=\max\left\{1,h_{\rm ch}^{-1},
 \sqrt8 Q\kappa\sqrt{1+V^2},
 Q\sqrt{2\kappa(1+V^2)/\rho}\right\},\quad
 C_0=\tfrac94Q\kappa^3(1+V^2)^2.
\end{equation}
Indeed, $Q\sqrt{2C/r}$ gives the last two terms and dominates
$D$, since $2(1+V)\le\sqrt8\sqrt{1+V^2}$.
For a prescribed nodal distance error $\eta$, it suffices to take
$K=\lceil\max\{\Gamma,(C_0/\eta)^{1/4}\}\rceil$.
Here $\alpha^{-1}$ measures sensitivity, $V=B/\alpha$ bounds
speed, and $C$ bounds acceleration. If a sharper bound for the
actual $\|x''\|$ is available, it can replace $C$ throughout the
tracking proof, retaining $D$ separately in the step restriction.

\begin{prop}[Local geometry and nodal accuracy]\label{prop:adaptive-geometry}
Work in one fixed chart. On every proposed interval beginning at $t$,
suppose positive continuous bounds $\kappa(t),C(t),D(t),r(t)$ justify
\eqref{eq:prediction-error}--\eqref{eq:newton-error} on its tube.
Let a continuous $a(t)>0$ bound its permitted time length. Fix $\eta>0$ with
\[
 \kappa(t)\eta\le1/16,\qquad
 \eta\le\kappa(t)r(t)^2,
\]
and define
\begin{equation}\label{eq:local-accuracy-scale}
 s_\eta(t)=\min\{a(t),D(t)^{-1},
             \eta^{1/4}\kappa(t)^{-1/4}C(t)^{-1/2}\}.
\end{equation}
Starting at the exact model zero, steps $h_k\le s_\eta(t_k)$
preserve nodal error at most $\eta$.
If every nonfinal step also satisfies $h_k\ge\theta s_\eta(t_k)$,
and $s_\eta(t)\le c s_\eta(t_k)$ on each accepted interval,
then
\begin{equation}\label{eq:adaptive-step-count}
 K\le1+\frac c\theta\int_0^1
 \max\{a^{-1},D,\eta^{-1/4}\kappa^{1/4}C^{1/2}\}\,dt.
\end{equation}
Here $0<\theta\le1$ and $c\ge1$.
\end{prop}
\begin{proof}
If the current error is at most $\eta$, prediction gives
\[
 e_p\le2\eta+Ch_k^2/2
 \le2\eta+\tfrac12\sqrt{\eta/\kappa}
 \le\sqrt{\eta/\kappa}\le r.
\]
Correction therefore gives $e_{k+1}\le\kappa e_p^2\le\eta$.
Induction starts at the exact model zero. Each nonfinal interval
contributes at least $h_k/(c s_\eta(t_k))\ge\theta/c$ to
$\int dt/s_\eta(t)$, proving the count.
\end{proof}

The bounds must hold on the whole proposed tube, not just at its
initial point. Local tests such as~\eqref{eq:tube-test} address this
requirement. Finite chart changes use the factors $Q$ from
Lemma~\ref{lem:charts}. With bounded correction cost $\mathcal W$,
the accepted-step work is $O(\mathcal W K)$; rejected trials and
the cost of establishing the bounds must also be counted.
Under uniformly bounded scale variation, the integral charges poor
conditioning only where it occurs. Using $C\le\kappa(1+V^2)$,
its accuracy term is at most
$\eta^{-1/4}\kappa^{3/4}\sqrt{1+V^2}$.
For residual accuracy $\varepsilon$, take $\eta\le\varepsilon/\Lambda$.
This concerns all corrected nodes; for endpoint accuracy alone,
fixed tracking tolerances followed by Newton refinement give the
smaller additional count~\eqref{eq:main-count}.

Multiplying $F$ by a nonzero constant changes neither $\kappa=L/\alpha$
nor $V=B/\alpha$. Thus these ratios avoid an artificial dependence
on the common scale of the equations. They still depend on the
chosen coordinate and target norms, which remain fixed.
Condition-length bounds for polynomial continuation are developed
in~\cite{ShubCondition,BeltranLeykin,BeltranLeykinRobust}.
Equation~\eqref{eq:adaptive-step-count} is an elementary smooth-map
bound with stated local hypotheses, not an identification with the
polynomial condition metric.

\subsection{Folds}
The extra equation below controls progress along the solution curve
without requiring the homotopy parameter to increase.

At a regular zero write $A=DF$ and choose the oriented unit vector
$q\in\ker A$. Pseudo-arclength prediction $z_p=z+hq$ uses the
corrector
\begin{equation}\label{eq:arclength}
 F(w)=0,\qquad \langle q,w-z_p\rangle=0.
\end{equation}
Its derivative at the current zero is the bordered matrix
$\mathcal A=\begin{pmatrix}A\\q^T\end{pmatrix}$, which is invertible since
$\ker A=\R q$. In orthonormal coordinates its singular values
are those of $A$ together with $1$.

For a compact $C^2$ regular arc, sufficiently small steps therefore
follow it through folds. Here is the needed finite-step justification.
In arclength coordinates, $Az''+D^2F[z',z']=0$ and
$\langle z',z''\rangle=0$, so $\|z''\|\le L/a$, where
$a$ is a lower bound for the smallest positive singular value of
$A$. The predictor error is $O(h^2)$; the correction plane meets
the same local arc at arclength advance $h+O(h^2)$ because its
derivative along that arc is $1+O(h)$. Compactness gives uniform
inverse bounds, chart margins, and Newton neighborhoods. Taking
the point error $O(h^2)$ and tangent error $O(h)$ preserves this
advance. Such tolerances are attainable by Newton correction and
continuity of the kernel direction, with orientation chosen by the
preceding tangent. Every sufficiently small accepted step advances
at least $h/2$. An arc of length $\ell$ thus requires
$O(1+\ell/h)$ steps to a target neighborhood, followed by Newton
correction on $t=1$. These constants concern the augmented system;
the monotone bounds above are not applied at a fold.

\section{Finding all regular solutions}\label{sec:all}
We keep the target fixed and vary its model, asking which solutions
are reachable and when the parameter search can stop. Throughout this
section all target zeros are regular, hence there are finitely many
zero orbits. A single selected arc reaches only one of them.

\subsection{Changing the model}
We first explain stability, then construct a family parametrized by
antipodal pairs and a geometric family for ham-sandwich cuts.

In a smooth family of regular homotopies with fixed regular target,
the selected endpoint is locally constant. Along a compact parameter path
$p(s)$, the total zero set modulo $G$ projects submersively to $s$,
because $D_{(x,t)}F$ is onto. At the horizontal boundaries,
invertibility of $D_xF$ gives lifts of $\partial_s$ tangent to those
boundaries. A partition of unity combines these lifts; compactness
makes their flow identify the zero curves. On the target boundary
the spatial component is zero, since the target is fixed and regular.
The initial endpoint therefore remains paired with the same target.
Thus endpoint changes require a singularity of the full homotopy
or a boundary degeneracy; spatial folds alone do not change the pairing.

\begin{prop}[A compatible model]\label{prop:compatible-model}
Let $g_p$ be a smooth family of one-orbit models and $q$ a regular
zero of $f$. If $g_{p_0}(q)=0$ and
\begin{equation}\label{eq:compatible-pencil}
 (1-t)Dg_{p_0}(q)+tDf(q)
\end{equation}
is invertible for $0\le t\le1$, then every $p$ in a neighborhood
of $p_0$ has a regular monotone branch from its model zero to $q$.
\end{prop}
\begin{proof}
At $p_0$ the branch is constant, $x(t)=q$. Compactness gives a
uniform inverse bound for~\eqref{eq:compatible-pencil}.
The implicit function theorem gives nearby graphs for all $t$;
uniqueness glues them on overlaps. At $t=0$ the graph is the
model zero near $q$, and at $t=1$ it is the fixed zero $q$.
\end{proof}

\begin{thm}[Antipodal pairs as model parameters]\label{cor:all-sphere}
Let $f:S^n\to\R^n$ be odd and $C^2$, and put
\[
 \mathcal P=\{[p]\in\mathbb{RP}^n:Df_p:p^\perp\to\R^n
                         \text{ is invertible}\}.
\]
For $[p]\in\mathcal P$, let $P_p$ be orthogonal projection onto $p^\perp$
and define
\begin{equation}\label{eq:pair-model}
 g_p(x)=Df_p(P_p x),\qquad F_p(x,t)=(1-t)g_p(x)+tf(x).
\end{equation}
This is a well-defined family on $\mathcal P$; each model has the unique
zero pair $\{p,-p\}$. Every regular target zero pair $[q]$ has a
neighborhood in $\mathcal P$ from which the selected branch is a regular
graph over $[0,1]$ ending at $[q]$.
\end{thm}
\begin{proof}
Oddness implies $Df_{-p}(v)=Df_p(v)$ for $v\in p^\perp$:
differentiate $f(-x)=-f(x)$. Since $P_{-p}=P_p$, one has $g_{-p}=g_p$.
Invertibility of $Df_p$ gives $\ker(Df_pP_p)=\R p$ and the stated zeros.
At a target zero $q$, $x(t)=q$ is constant and
$D_xF_q(q,t)=Df_q$ for every $t$.

For a uniform neighborhood argument choose coordinates $z$ near $q$,
with $z=0$ at $q$, and write $H(z,t,p)=F_p(x(z),t)$.
Let $A=D_zH(0,t,q)$; this invertible matrix is independent of $t$.
Continuity and compactness of $[0,1]$ give $r>0$ and a neighborhood
$W$ of $q$ such that, for $\|z\|\le r$, $p\in W$, and all $t$,
\[
 \|I-A^{-1}D_zH(z,t,p)\|\le\tfrac12,
 \qquad \|A^{-1}H(0,t,p)\|\le r/4.
\]
Thus $z\mapsto z-A^{-1}H(z,t,p)$ contracts the closed radius-$r$ ball
into itself. Its unique zero depends continuously on $(t,p)$ and is
$C^1$ in $t$; $D_zH$ remains invertible. Shrink $W$ so the initial
zero is $p$. At $t=1$, uniqueness in the same ball forces the zero
to be $q$. Quotienting by antipodality proves the assertion.
\end{proof}

The parameter is just an antipodal pair: use a closed hemisphere and
identify opposite points on its boundary. Formula~\eqref{eq:pair-model}
automatically chooses the target coordinates; this adjustment matters,
and an arbitrary prescribed frame rule is not covered by the theorem.
No unknown solution is needed to evaluate $g_p$: it uses $Df$ at the
chosen start. Singular derivatives are excluded, but every regular
target lies in the admissible open set. With only $C^1$ regularity of
$f$ the same contraction proof gives reachability and continuous
parameter dependence; the numerical error estimates use $C^2$.

A search may interleave dense starts with increasing computation budgets
and decreasing steps. Each regular target has an open set of successful
starts with a common positive admissible step, so every target is
approximable. Interleaving prevents a singular run from blocking later
starts. Finite stopping requires additional information, such as a
solution count or quantitative bounds on successful neighborhoods.
The parameter space in Theorem~\ref{cor:all-sphere} is an open subset
of the space of antipodal pairs. Its simple geometry makes systematic
coverage possible without searching through arbitrary linear maps.
For mass partitions, Theorem~\ref{cor:all-ham} gives a geometric
alternative using Gaussian centers.

Other multiple-solution methods include polynomial and real
monodromy~\cite{DuffEtAl,HauensteinRegan} and deflation~\cite{FarrellDeflation};
their completeness conditions differ from our model-parameter search.

\subsection{Finite parameter regions}\label{sec:regions}
A finite search is possible when the parameter space splits into
finitely many regions with constant answers. Piecewise-linear maps
allow us to describe such regions by finite systems of inequalities.

Let $K$ be a fixed finite equivariant triangulation of the cylinder
over a closed free $G$-PL $n$-manifold, with both boundary slices
as subcomplexes. On a compact semialgebraic
parameter space $P$, prescribe continuous semialgebraic vertex
values $a_v(p)\in V$, equivariantly in $v$, and extend affinely
on each simplex. Denote the resulting homotopy by $F_p$.
Assume a unique initial zero orbit for every $p$ and a fixed target
whose zeros avoid its $(n-1)$-skeleton. Define
\begin{equation}\label{eq:PL-discriminant}
 \Delta_K=\{p:F_p^{-1}(0)\cap K^{(n-1)}\ne\varnothing\},
 \qquad U=P\setminus\Delta_K.
\end{equation}

\begin{thm}[Finite endpoint regions]\label{thm:finite-regions}
The set $\Delta_K$ is closed and semialgebraic. Its complement
has finitely many connected components, and the selected target
orbit is constant on each. For each target orbit $q_i$, its region
$B_i\subset U$ is relatively open and semialgebraic with finitely
many connected components, and $\partial_PB_i\subset\Delta_K$.
Here $\partial_P$ is the boundary relative to $P$. If $P$ is an
$r$-dimensional semialgebraic manifold, possibly with boundary,
these relative boundaries
have finite smooth stratifications of dimension at most $r-1$.
\end{thm}
\begin{proof}
A zero on a face $[v_0,\ldots,v_k]$, $k<n$, is equivalent to
the feasibility of
\[
 \lambda_j\ge0,\quad \sum_j\lambda_j=1,\quad
 \sum_j\lambda_j a_{v_j}(p)=0.
\]
Semialgebraic projection and a finite union show that $\Delta_K$
is semialgebraic. Compactness of the barycentric simplex shows it
is closed. Its complement has finitely many semialgebraic components.

For $p\in U$, every zero on an $n$-face is interior and regular
by Proposition~\ref{prop:pl}; its positive barycentric coordinates
persist locally. A zero-free face has positive distance from zero,
also preserved locally. Thus zero-containing faces and their
segment incidences are locally constant. The fixed target marks
and unique initial mark select the same endpoint. Each $B_i$ is
a union of finitely many components of $U$, so its boundary lies
in $\Delta_K$. Finite stratification and the dimension bound are
the relative-boundary properties of open semialgebraic subsets
of a manifold.
\end{proof}

These triangulation-dependent walls can separate regions with the
same endpoint; $\Delta_K$ need not have empty interior. For a general
semialgebraic parameter space use a finite stratification and its incidences.

General semialgebraic families admit finite compatible partitions
by Hardt's theorem~\cite{Hardt}. Being PL in $x$ for each parameter
does not imply joint PL dependence.

\begin{figure}[!t]
\centering
\includegraphics[width=.98\linewidth]{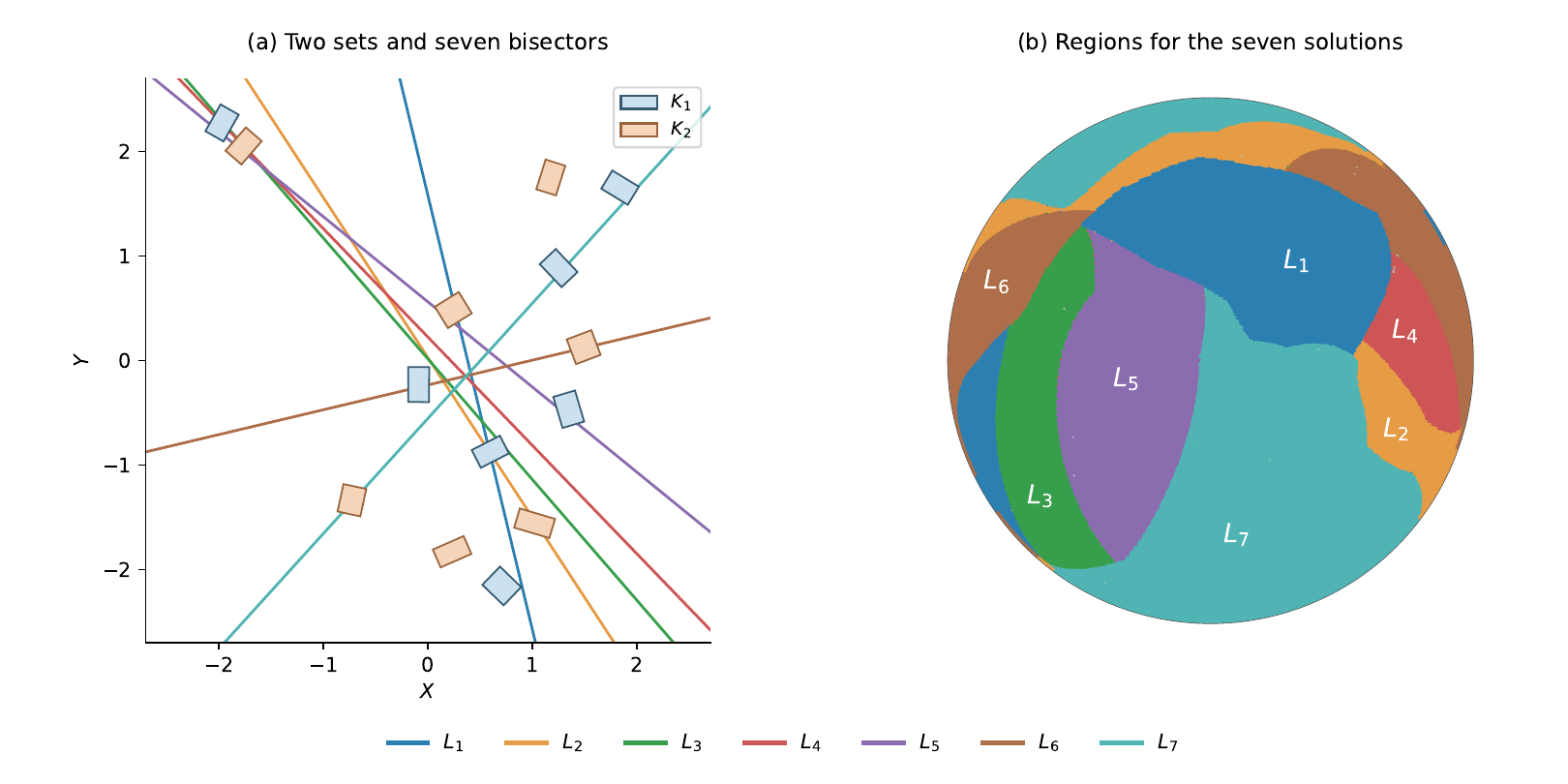}
\caption{Left: seven common bisectors of the two polygonal sets.
Right: the corresponding seven solution regions. Matching colors
identify each bisector and its region.}
\label{fig:polygon-regions}
\end{figure}
\subsection{Region growing}\label{sec:region-growing}
Reachability leaves a choice of how to explore the parameter space.
Region growing is one option; a particular application's symmetries,
explicit region descriptions, or other structure may give a faster
enumeration. The following result analyzes this option without claiming
its optimality. Grow a region from a seed parameter, then restart in
the unexplored part. This is flood fill with restarts (informally,
the \emph{rodent method}), implemented by graph
search~\cite{Tarjan}. Related constructions include component
labeling~\cite{HoshenKopelman} and covering a solution manifold
by overlapping continuation neighborhoods~\cite{Henderson}.

\begin{prop}[Finite exhaustion]\label{prop:region-growing}
Suppose a finite triangulation $\mathcal T$ of $P$ represents a
closed exceptional set $\Delta$ by a subcomplex $\mathcal D$, and
the endpoint label on $U=P\setminus\Delta$ is locally constant.
Make a graph with one vertex for every simplex outside $\mathcal D$,
joining two whenever one is a proper face of the other. If the
graph has $C$ vertices, $E$ edges, and $b$ components, repeated
graph search needs one continuation per component and $O(C+E)$
additional work to find all endpoints reachable from $U$.
\end{prop}
\begin{proof}
In the barycentric subdivision, discard coordinates corresponding
to simplices of $\mathcal D$ and renormalize the rest. Their sum
is positive on $U$. The straight homotopy stays in the same simplex
and outside $\Delta$, giving a deformation retraction onto the
full subcomplex of the remaining barycenters. Its one-skeleton is
the stated graph, so its components are exactly those of $U$.
One continuation from a relative-interior point of a seed simplex
labels its whole component. Search visits each vertex and adjacency
list once; restarting exhausts the graph. Duplicate endpoint labels
are merged in the answer list.
\end{proof}

A semialgebraic triangulation compatible with $\Delta_K$ provides
this hypothesis for Theorem~\ref{thm:finite-regions}. If labeling
the $j$th component costs $T_j$, including continuation and endpoint
identification, the total after the graph is constructed is
\begin{equation}\label{eq:region-growing-cost}
 O(C+E)+\sum_{j=1}^bT_j.
\end{equation}
Several components may carry the same label; no connectedness
assumption on their union is used. To obtain \emph{all} target zeros, every target
must additionally be reachable in the chosen parameter domain.

Adaptive exploration can use validated overlapping neighborhoods,
obtained from contraction tests as in~\eqref{eq:tube-test} with
model parameters included in the perturbation bounds.

Agreeing labels at finitely many samples do not justify labeling
an entire cell. Near walls the validated neighborhoods may shrink
without a finite stopping time. Finite exhaustion therefore uses
a compatible decomposition or an explicit resolution assumption:
if every label contains a parameter ball of radius $r_*>0$, any
net of covering radius less than $r_*$ meets every label.
In a bounded $r$-dimensional parameter chart a regular mesh has
$O(h^{-r})$ cells, making this approach most suitable for small
parameter dimension. Projective chart seams must use the prescribed
identifications. Random restarts use a normalized vector of
independent standard normals, folded onto the hemisphere~\cite{Muller}.
Rotational invariance gives uniform projective measure. Rejecting explored
points samples the remainder uniformly; for remaining fraction $a>0$
the expected number of trials is $1/a$. This does not replace a finite
covering or stopping argument.

Figure~\ref{fig:polygon-regions} illustrates the method for two
polygonal sets. Each of the seven colored lines on the left bisects
the area of both sets. On the right, initial parameters are colored
according to the solution reached by continuation, using the same
colors as the corresponding lines. The disk represents a hemisphere
with opposite boundary points identified. The computation samples
antipodal pairs over the whole sphere without duplicate runs;
the hemisphere is used only for display. Numerical data and details
of the chosen model family are provided in the computational archive.

\FloatBarrier
\section{Mass partitions and finite point sets}\label{sec:partitions}
We give explicit test maps and decode their zeros. For mass partitions
we also retain useful models arising from simple geometric inputs.

A Borel probability measure is \emph{hyperplane-null} if every affine
hyperplane has measure zero. A positive Schwartz density is a
strictly positive smooth function of integral one, all of whose
derivatives decrease faster than any inverse power at infinity.

\subsection{Ham-sandwich cuts}
We justify the spherical description from the introduction and
identify the local geometry that makes a bisector regular.

Use $S^d=\{(a,w)\in\R\times\R^d:a^2+\|w\|^2=1\}$.
For $d$ hyperplane-null probability measures on $\R^d$, set
\begin{equation}\label{eq:ham-map}
 f_i(a,w)=\mu_i\{x:w\cdot x\ge a\}-\tfrac12,
 \qquad (a,w)\in S^d,\quad w\ne0.
\end{equation}
At the poles set $f_i(1,0)=-1/2$ and $f_i(-1,0)=1/2$.
Dominated convergence at finite hyperplanes and tightness at the
poles prove continuity. Boundary nullity gives oddness, and zeros
are exactly the common bisecting hyperplanes, modulo orientation.
For $w\ne0$, unit-normal coordinates are
$(a/\|w\|,w/\|w\|)$. Below the offset in these coordinates is
again denoted by $a$.

Take model measures $\nu_i=\mathcal N(c_i,I_d)$, where
$c_1=0$ and $c_i=e_{i-1}$ for $2\le i\le d$. Their map is
\begin{equation}\label{eq:ham-model}
 g_i(a,w)=\Phi\bigl((w\cdot c_i-a)/\|w\|\bigr)-\tfrac12,
\end{equation}
where $\Phi$ and $\phi$ denote the standard normal distribution
and density. A Gaussian is bisected precisely by hyperplanes
through its center. These centers span $x_d=0$, so the only zeros
are $\pm(0,e_d)$. At $(0,e_d)$,
\[
 Dg_i[\delta a,\delta w]=\phi(0)(c_i\cdot\delta w-\delta a),
 \qquad \delta w_d=0.
\]
The first row forces $\delta a=0$ and the remaining rows force
$\delta w_1=\cdots=\delta w_{d-1}=0$. The model is regular,
also for $d=1$.

For a positive Schwartz density $\rho_i$, the map is smooth.
Away from the poles this follows by differentiating under the
integral; near the poles rapid decay makes the tails and all
their derivatives tend to zero faster than any power of $\|w\|$.
In unit-normal coordinates $H=\{x:u\cdot x=a\}$, put
\[
 m_i(H)=\int_H\rho_i\,dS>0,\qquad
 c_i(H)=m_i(H)^{-1}\int_Hx\rho_i(x)\,dS.
\]
Writing $x=z+su$, $z\perp u$, and differentiating the moving
lower endpoint of the half-line integral gives
\begin{equation}\label{eq:ham-derivative}
 Df_i[\delta a,\delta u]
 =m_i(H)(c_i(H)\cdot\delta u-\delta a),\qquad\delta u\perp u.
\end{equation}
Rapid decay justifies differentiation. Since all centroids belong
to $H$, a kernel vector satisfies
$(c_i-c_1)\cdot\delta u=0$ for $i\ge2$ and
$\delta a=c_1\cdot\delta u$. Consequently $Df$ is invertible
exactly when the $d$ section centroids are affinely independent.

\begin{lemma}[Section-centroid criterion]\label{lem:centroid}
For positive Schwartz densities, a common bisector is a regular
zero precisely when its $d$ section centroids are affinely independent.
In unit-normal coordinates its Jacobian is
$J=\diag(m_i)C_H$, where $C_H$ has rows $(-1,c_i^TE)$ and
$E$ is an orthonormal basis matrix for $u^\perp$. In particular,
$\sigma_{\min}(J)\ge\min_i m_i\,\sigma_{\min}(C_H)$.
\end{lemma}
\begin{proof}
The derivative and kernel calculation above give the factorization
and equivalence. The singular-value bound follows by applying
$\|\diag(m_i)z\|\ge(\min_i m_i)\|z\|$ to $z=C_Hv$.
\end{proof}

\begin{cor}\label{cor:ham}
Every $d$ hyperplane-null probability measures in $\R^d$ have a
common bisecting hyperplane. For positive Schwartz densities,
the Gaussian model above has one regular antipodal zero orbit.
A homotopy satisfying Theorem~\ref{thm:continuation} reaches a
common bisector. For the mixture homotopy
$\mu_{i,t}=(1-t)\nu_i+t\mu_i$, the numerical bounds of
Theorem~\ref{thm:numerical} apply whenever the section centroids
remain affinely independent along the selected arc.
\end{cor}
\begin{proof}
Apply Proposition~\ref{prop:continuous} to the odd map
\eqref{eq:ham-map} and the regular model~\eqref{eq:ham-model}.
Its zero is not a pole and hence gives a finite bisector.
The smooth continuation statements follow from the derivative
criterion and Theorems~\ref{thm:continuation}--\ref{thm:numerical}.
\end{proof}

For the mixture homotopy, $\|F_t\|\le\sqrt d$.
An alternative model uses centered regular-simplex centers in
$e_d^\perp$, normalized by $\sum_i c_i=0$ and
$\sum_i c_ic_i^T=I-e_de_d^T$. Direct expansion gives
\begin{equation}\label{eq:simplex-model-condition}
 \|Dg[\delta a,\delta u]\|^2
 =\phi(0)^2(d\,\delta a^2+\|\delta u\|^2).
\end{equation}
Thus the model has smallest singular value $\phi(0)$ independently
of $d$; this is an initial bound, not a bound along the homotopy.

\paragraph{Accuracy for measures.}
For the normalized measures above, the mass error of a hyperplane is
\begin{equation}\label{eq:mass-error}
 E_\mu(H)=\max_i\left|\mu_i(H^+)-\tfrac12\right|
          =\|f_\mu(H)\|_\infty.
\end{equation}
Since hyperplanes have zero mass, the opposite half-space has
the same error. Thus the residual estimate in
Proposition~\ref{prop:oracle-work} gives an $\varepsilon$-bisector
in mass. For unnormalized finite positive measures, normalization
makes this an error of at most $\varepsilon$ times each total mass.
Here $\mathcal W(\tau)$ includes the half-space integrals and
the section integrals determining their derivatives in
\eqref{eq:ham-derivative}, to the required accuracy.

If the input measures are approximated by probability measures
$\widetilde\mu_i$, the appropriate error for this problem is
their discrepancy on half-spaces:
\[
 \Delta_{\mathcal H}(\mu,\widetilde\mu)
   =\max_i\sup_{H^+\ {\rm half\text{-}space}}
                |\mu_i(H^+)-\widetilde\mu_i(H^+)|.
\]
Directly from the triangle inequality,
\begin{equation}\label{eq:measure-error-budget}
 E_\mu(H)\le E_{\widetilde\mu}(H)
                +\Delta_{\mathcal H}(\mu,\widetilde\mu).
\end{equation}
Consequently discrepancy at most $\varepsilon/2$ and a certified
mass error at most $\varepsilon/2$ give an $\varepsilon$-bisector
of the original measures. For densities, the discrepancy is at most
$\tfrac12\max_i\|\rho_i-\widetilde\rho_i\|_{L^1}$: their difference
has integral zero, so its positive and negative parts have equal
integrals. Constructing the approximate measures is a separate
input cost. Discrepancy controls the final mass error; derivative
and conditioning bounds must still be checked for the homotopy
actually tracked.

\begin{thm}[Reachability of all regular ham-sandwich cuts]
\label{cor:all-ham}
Let $d$ positive Schwartz densities have only regular common
bisectors. In the family of models consisting of $d$ standard
Gaussian measures with affinely independent centers, every target
bisector is reached by a regular monotone branch for a nonempty
open set of center configurations.
\end{thm}
\begin{proof}
Fix a regular target bisector $H$ and choose its section centroids
$c_i$ as the Gaussian centers. By~\eqref{eq:ham-derivative}, they
are affinely independent, so this model has $H$ as its unique zero
orbit. In unit-normal tangent coordinates let $C_H$ have rows
$(-1,c_i^TE)$. Then
$Df(H)=\diag(m_i)C_H$ and $Dg(H)=\phi(0)C_H$.
Their pencil is
\[
 \diag\bigl((1-t)\phi(0)+tm_i\bigr)C_H,
\]
which is invertible for every $t$. Apply
Proposition~\ref{prop:compatible-model}.
\end{proof}

\subsection{Orthogonal pairwise four-partitions}
Here the unknown is an orthogonal frame. Individual bisection
determines the offsets, leaving the pairwise balance equations.

The pairwise orthogonal partition theorem was stated by
Makeev~\cite{Makeev}; the equivariant model used here is
from~\cite{MusinOrth}.
Let $B_d$ be the signed permutation group and let $W_2^+$ be the
space of symmetric $d\times d$ matrices with zero diagonal.
For a square matrix $A$, $\off(A)$ is obtained by setting its
diagonal entries to zero.
Use the left actions
$P\cdot U=UP^{-1}$ on $O(d)$ and $P\cdot C=PCP^{-1}$ on
$W_2^+$. The first is free and both spaces have dimension
$n=\binom d2$. For a diagonal matrix
$A_0=\diag(\alpha_1,\ldots,\alpha_d)$ with distinct entries, set
\begin{equation}\label{eq:orth-model}
 g(U)=\off(U^TA_0U).
\end{equation}
Since signed permutations preserve diagonal matrices,
$g(UP^{-1})=\off(PU^TA_0UP^{-1})=Pg(U)P^{-1}$.
Thus the map is equivariant, and $g(U)=0$ means that the columns of
$U$ are eigenvectors of $A_0$. Its zeros are exactly the signed
coordinate eigenbases, one $B_d$-orbit. For $\Omega^T=-\Omega$,
\begin{equation}\label{eq:orth-model-derivative}
 Dg_I[\Omega]_{ij}=(\alpha_i-\alpha_j)\Omega_{ij},\quad i<j,
\end{equation}
by differentiating $e^{-s\Omega}A_0e^{s\Omega}$.
These are independent tangent coordinates with nonzero coefficients.
Equivariance gives regularity at all zeros.

Let $\rho$ be a positive Schwartz probability density. For each
unit $u$ there is a unique median offset $a(u)$, since the
half-space mass decreases strictly from $1$ to $0$ with derivative
$-\int_{u\cdot x=a}\rho\,dS<0$. The implicit function theorem
gives smoothness, and $a(-u)=-a(u)$.
For $U=(u_1,\ldots,u_d)$ put
$s_i(x)=\sgn(u_i\cdot x-a(u_i))$ and
\begin{equation}\label{eq:pair-map}
 f_{ij}(U)=\int s_i(x)s_j(x)\rho(x)\,dx,\qquad i<j.
\end{equation}
Signs and permutations of the columns give the stated equivariance.
Orthogonal coordinates and translation by the offsets express the
quadrant masses as integrals over fixed orthants, proving smoothness.
The median equations give $\int s_i\rho=0$, and therefore
\begin{equation}\label{eq:quadrants}
 \mu\{s_i=\epsilon_i,s_j=\epsilon_j\}
 =\tfrac14(1+\epsilon_i\epsilon_j f_{ij}),
 \qquad\epsilon_i,\epsilon_j\in\{-1,1\}.
\end{equation}
Zeros of $f$ are precisely the pairwise four-partitions.
This asserts equality of the four parts for every pair; it does
not assert equality of all $2^d$ orthants.

For evaluating the Jacobian, let $H_i=\{u_i\cdot x=a(u_i)\}$
and let $c_i$ be its section centroid. Differentiating the median
equation gives $Da(u_i)[\delta u_i]=c_i\cdot\delta u_i$.
The jump of a sign function is $2$, so
\begin{equation}\label{eq:pair-derivative}
 \begin{split}
 Df_{ij}[\delta U]={}&2\int_{H_i}s_j(x)\,
               \delta u_i\cdot(x-c_i)\rho(x)\,dS\\
 &+2\int_{H_j}s_i(x)\,
               \delta u_j\cdot(x-c_j)\rho(x)\,dS.
 \end{split}
\end{equation}
Use local orthogonal coordinates and differentiate the integrals
on each side of the moving boundaries. Schwartz decay justifies
dominated differentiation, and $H_i\cap H_j$ has surface measure
zero in either hyperplane. The boundary terms give the displayed
formula.
Setting $\delta U=U\Omega$ gives the square Jacobian in the
$\binom d2$ independent entries of $\Omega$. Prediction and
correction can be performed in charts $\Omega\mapsto U\exp\Omega$.

\begin{cor}\label{cor:orth}
For $d\ge2$, every hyperplane-null probability measure on $\R^d$
admits $d$ mutually orthogonal affine hyperplanes such that every
pair cuts it into four equal parts. For positive Schwartz densities,
\eqref{eq:orth-model} and~\eqref{eq:pair-map} give a continuation
problem satisfying the model assumptions of
Theorem~\ref{thm:continuation}. The numerical conclusion holds
under its remaining rank assumptions and spatial invertibility
along the selected arc.
\end{cor}
\begin{proof}
For positive Schwartz densities, Proposition~\ref{prop:continuous}
gives a zero of~\eqref{eq:pair-map}; equation~\eqref{eq:quadrants}
proves the partition assertion. For a general hyperplane-null
measure $\mu$, choose finite positive Gaussian mixtures $\mu_j$
converging weakly to $\mu$. Finitely supported approximations
followed by Gaussian convolution provide such a sequence.

Choose a ball which is a continuity set of $\mu$ and has mass
greater than $1/2$. It eventually has $\mu_j$-mass greater than
$1/2$, so all median hyperplanes of the constructed partitions
meet this ball. Their offsets are bounded. Compactness of $O(d)$
and of the offset interval gives convergent frames and offsets.
The limiting frame is orthogonal. On any fixed ball the defining
affine functions converge uniformly, so the moving quadrants lie
between small inward and outward shifts of the limiting quadrants.
Weak convergence applies to shifts chosen as continuity sets.
Let the shifts tend to zero, using nullity of the limiting
hyperplanes, and then let the ball grow, using tightness.
All quadrant masses converge to $1/4$.

The regularity and equivariance of the model were checked in
\eqref{eq:orth-model-derivative}. Applying the two continuation
theorems proves the numerical statement.
\end{proof}

For either application a regular target permits an interior
perturbation by Lemma~\ref{lem:regularize}. A prescribed linear
homotopy need not already be regular or monotone.

\subsection{Finite point sets and smoothing}\label{sec:points}
The objective here is to turn an approximate smooth partition into
an exact statement about point counts.

Let $P_1,\ldots,P_d\subset\R^d$ be finite nonempty sets,
$n_i=|P_i|$, and $N=\sum_i n_i$. A \emph{discrete bisector}
has at most $\lfloor n_i/2\rfloor$ points of each color in either
open half-space. Points on the hyperplane are excluded from both
counts. Small-ball smoothing
and reduction to odd cardinalities are classical devices;
see~\cite[Section~3.1]{Matousek}.

\paragraph{Regular target zeros.}
We identify conditions under which the smoothed problem has
isolated, stable solutions corresponding to the discrete cuts.

Replace each point by $\mathcal N(p,\sigma^2I_d)$, with equal
weights within each color. For $H=\{u\cdot x=a\}$, $\|u\|=1$,
write $s_p=u\cdot p-a$. The mass imbalance and its derivative are
\begin{align}
 f_{i,\sigma}(a,u)&=\frac1{n_i}\sum_{p\in P_i}\Phi(s_p/\sigma)-\tfrac12,
 \label{eq:atomic-gaussian}\\
 Df_{i,\sigma}[\delta a,\delta u]
 &=\frac1{n_i\sigma}\sum_{p\in P_i}\phi(s_p/\sigma)
                  (p\cdot\delta u-\delta a),\qquad\delta u\perp u.
 \label{eq:atomic-derivative}
\end{align}
These finite sums require no multidimensional integration. They
give the smooth odd target of Section~\ref{sec:partitions}.

\begin{prop}[Discrete limit]\label{prop:atomic-limit}
If $\sigma_j\to0$ and
$\max_i|f_{i,\sigma_j}(a_j,u_j)|\le\delta_j\to0$, a subsequence
of the hyperplanes converges to a discrete bisector.
\end{prop}
\begin{proof}
If $\|p\|\le R$ and $\delta<1/2$, monotonicity of $\Phi$ gives
\begin{equation}\label{eq:offset-bound}
 |a|\le R+\sigma\Phi^{-1}(1/2+\delta).
\end{equation}
For instance, when $a>R$ every summand is at most
$\Phi((R-a)/\sigma)$, while their average is at least $1/2-\delta$.
Thus offsets are bounded and a subsequence converges, using compactness
of the unit sphere. Every point strictly on one side of the limiting
plane contributes asymptotically its full weight to that side.
The total is at most $1/2$ in each color. Integrality gives the
discrete bound on both sides.
\end{proof}

An exact smooth cut need not bisect the original points;
Figure~\ref{fig:point-smoothing} illustrates the required reconstruction.

From now through Theorem~\ref{thm:point-rounding}, assume that
the sets are pairwise disjoint, all $n_i$ are odd, and their union
is in \emph{affine general position}: every subset of at most
$d+1$ points is affinely independent. Every discrete bisector
contains a point of each color, since an odd count cannot split
equally off the plane. General position permits at most $d$ points
on a hyperplane. Hence it contains exactly one $q_i\in P_i$,
and the remaining points of each color split equally. The set
$\mathcal B$ of unoriented discrete bisectors is therefore finite.

Parity is easily arranged: delete one point from each even-size
set. If its original size is $2k$, a bisector of the remaining
$2k-1$ points leaves at most $k-1$ on either side. Restoring the
deleted point preserves bisection of the original set. This does
not impose general position.

Compact smoothing first identifies the discrete cuts exactly.
The Gaussian analogue uses a limiting argument and gives regular
nearby cuts; the reconstruction theorems then allow a nonzero
numerical residual. Define the slab parameter
\begin{equation}\label{eq:slab-parameter}
 r_* =\min_{\substack{T\subset\bigcup_iP_i\\|T|=d+1}}
       \ \min_{\|u\|=1,\ a\in\R}\max_{p\in T}|u\cdot p-a|,
\end{equation}
with $r_*=\infty$ when $N=d$. For $N>d$ the inner minimum
is attained with bounded $a$ and is positive by general position;
the finite outer minimum is positive. Thus a closed slab of
half-width less than $r_*$ contains at most $d$ input points.

\begin{thm}[Compact smoothing]\label{thm:compact-points}
Let $\vartheta$ be a smooth radial probability density supported in
the closed unit ball and positive in its interior. Replacing each
point by $\sigma^{-d}\vartheta((x-p)/\sigma)$, with $0<\sigma<r_*$,
gives exactly the common bisectors $\mathcal B$, all regular.
The equality of solution sets also holds for uniform solid balls.
\end{thm}
\begin{proof}
The slab bound permits at most $d$ support balls to meet the plane.
An odd color requires an interior intersection, since otherwise its
unnormalized half-space mass would be an integer. Thus exactly one ball
per color meets the plane. Writing $n_i=2k_i+1$, its mass equation is
$K_i+Q_i=k_i+1/2$, with $K_i$ integral and $0<Q_i<1$. Consequently
$K_i=k_i$ and $Q_i=1/2$. Radial symmetry and positivity force the plane
through all $d$ centers, bisecting the original points. Conversely a
discrete bisector meets exactly these central balls and bisects the
smoothed measures.

If $q_\vartheta(0)>0$ is the central density of a one-dimensional
marginal, the derivative row is
$q_\vartheta(0)(q_i\cdot\delta u-\delta a)/(n_i\sigma)$.
The differences $q_i-q_1$ span $u^\perp$, proving invertibility.
The same counting argument works for uniform balls; their map is
smooth near these central cuts, though not globally $C^2$ in
every dimension.
\end{proof}

\begin{thm}[Gaussian smoothing]\label{thm:gaussian-points}
For all sufficiently small $\sigma>0$, Gaussian common bisectors
are in bijection with $\mathcal B$, modulo orientation, and all
are regular. The cut $H_\sigma$ corresponding to $H_0\in\mathcal B$
converges to $H_0$. If the other input points have distance at
least $\gamma>0$ from $H_0$, then in fixed local coordinates
\begin{equation}\label{eq:gaussian-displacement}
 \dist(H_\sigma,H_0)=O\bigl(\sigma e^{-\gamma^2/(8\sigma^2)}\bigr).
\end{equation}
When $N=d$, the cuts coincide for every bandwidth.
\end{thm}
\begin{proof}
At $H_0$, its signed distances $y_i=u\cdot q_i-a$ to the
distinguished points form local coordinates: their derivative
is invertible by the preceding kernel argument. Shrink the
neighborhood so that all other points keep their sides and have
distance at least $\gamma/2$. In the rescaled coordinates
$y=\sigma z$, the equations are
\begin{equation}\label{eq:gaussian-rescaled}
 n_if_{i,\sigma}(H(\sigma z))
 =\Phi(z_i)-\tfrac12+E_i(\sigma,z)=0,
 \qquad |E_i|\le(n_i-1)\Phi(-\gamma/(2\sigma)).
\end{equation}
The off-plane tails and their derivatives decay exponentially;
on bounded $z$-neighborhoods they extend smoothly by zero at
$\sigma=0$. The $z$-Jacobian there is $\phi(0)I_d$.
The implicit function theorem gives a unique nearby solution
$z(\sigma)$ with invertible Jacobian. The bound on $E_i$ and
bounded derivative of $\Phi^{-1}$ near $1/2$ give
\eqref{eq:gaussian-displacement} and regularity.

Any additional zeros with $\sigma_j\to0$ would, by compactness and
Proposition~\ref{prop:atomic-limit}, converge after orientation to some
$H_0\in\mathcal B$. Its off-plane points eventually retain their sides
and have distance at least $\gamma/2$. Their tails tend to zero, so the
equations imply $\Phi(y_i(H_j)/\sigma_j)\to1/2$, hence
$y_i(H_j)/\sigma_j\to0$. These zeros eventually enter the same
implicit-function uniqueness neighborhood, a contradiction. Finiteness
of $\mathcal B$ gives a common threshold. For $N=d$ the equations are $y_i=0$.
\end{proof}

In a chart at $H_0$ whose tangent coordinates are
$(\delta a,E\xi)$, let $C_{H_0}$ have rows $(-1,q_i^TE)$.
The same proof and~\eqref{eq:atomic-derivative} give
\begin{equation}\label{eq:atomic-endpoint-condition}
 \sigma J_\sigma(H_\sigma)\longrightarrow
           \diag(\phi(0)/n_i)C_{H_0}.
\end{equation}
Thus the inverse target Jacobian is $O(\sigma)$ for fixed data.
This endpoint bound gives no lower singular-value bound in the
interior of a chosen homotopy.

\subsection{Reconstruction}
The computed smooth bisector usually misses the input points.
We show how the closest points determine an exact discrete answer
once the bandwidth and residual are sufficiently small.

\begin{thm}[Approximate cuts]\label{thm:point-rounding}
Put $n_{\max}=\max_i n_i$ and fix $0\le\delta_0<1/(2n_{\max})$.
For all sufficiently small $\sigma$, every hyperplane with
$\max_i|f_{i,\sigma}(H)|\le\delta_0$ has a unique closest point
$q_i$ of each color, and the affine hyperplane through them is
an exact discrete bisector.
\end{thm}
\begin{proof}
Otherwise choose counterexamples with $\sigma_j\to0$.
The offset bound~\eqref{eq:offset-bound} gives a convergent
subsequence with limit $H_0$. If one open side of $H_0$ contained
at least $(n_i+1)/2$ points, their smooth mass would tend to at
least $1/2+1/(2n_i)>1/2+\delta_0$, a contradiction.
Hence $H_0$ is a discrete bisector and contains exactly one point
$q_i$ per color. All others have positive distance from it, so
eventually these same $q_i$ are uniquely closest to $H_j$.
Their affine hull is $H_0$, contradicting the counterexample.
\end{proof}

Theorem~\ref{thm:point-rounding} gives a threshold $\sigma_{\rm rec}>0$
for the following stopping rule. After the parity reduction,
choose a bandwidth and compute residual at most $1/(8n_{\max})$.
Take the closest point of each color, form their plane, and check
the counts in the \emph{original} sets. Return it if all counts
are at most $\lfloor |P_i|/2\rfloor$; otherwise halve the bandwidth
and repeat. Once the residual is attained at each scale, at most
\begin{equation}\label{eq:bandwidth-reductions}
 \max\{0,\lceil\log_2(\sigma_{\rm start}/\sigma_{\rm rec})\rceil\}
\end{equation}
halvings suffice. The threshold need not be known to run the procedure: the count check
is the stopping test. For rational
coordinates it uses exact linear algebra and determinant signs.
A fresh model-to-target continuation may be used at each scale;
regularity of one branch across all bandwidths is not assumed.
The count check verifies a returned partition; it does not ensure
that every attempted continuation attains the required residual.

Figure~\ref{fig:point-smoothing} shows two five-point sets.
The smooth cuts each leave three original points on one side
and two on the other. The line through the two closest points
has two of each color on either side and one on the line.
\begin{figure}[htbp]
\centering
\includegraphics[width=.88\linewidth]{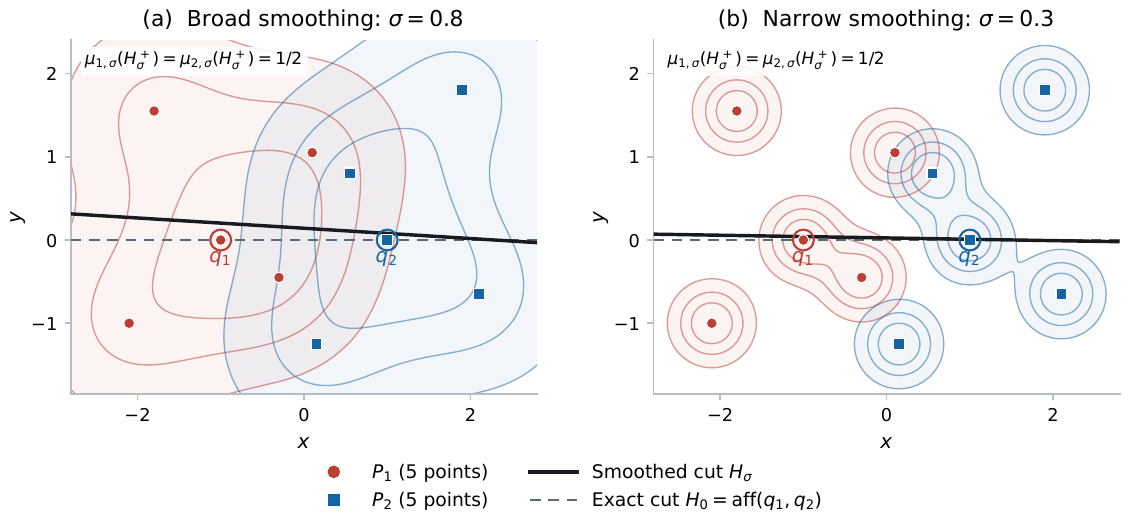}
\caption{Gaussian smoothing and reconstruction. Solid lines are
computed smooth bisectors at two bandwidths. The dashed line
through the circled closest points is an exact discrete bisector.}
\label{fig:point-smoothing}
\end{figure}
\FloatBarrier

\subsection{Complexity: the number of points, step, and accuracy}
\label{sec:complexity}
We count arithmetic and evaluations of $\Phi,\phi$ at unit cost,
with fixed dimension and bandwidth during each continuation.

\paragraph{Cost of one step.}
For each point, compute $u\cdot p-a$ and accumulate its contribution
to the value and ambient derivative row in
\eqref{eq:atomic-gaussian}--\eqref{eq:atomic-derivative}.
Convert the accumulated rows to a tangent basis, then solve the
linear systems for prediction and correction. Including the
dimension explicitly, Algorithm~\eqref{eq:algorithm} costs
$O(Nd+d^3)$ per step: the first term comes from the point sums,
the second from basis conversion and dense linear algebra.
With our fixed-dimension convention this is $O(N)$, as are
finding the closest points and checking the reconstructed plane.

With $h=1/K\le h_0$, Theorem~\ref{thm:numerical} gives
\begin{equation}\label{eq:step-accuracy-work}
 \max_k\dist(y_k,x(t_k))\le C_0h^4,\qquad T=O(NK).
\end{equation}
Consequently, nodal distance accuracy $\eta$ costs
\begin{equation}\label{eq:conditional-point-work}
 O\left(N\left[1+\max\{\Gamma,(C_0/\eta)^{1/4}\}\right]\right).
\end{equation}
Uniform bounds on $K$ give a uniformly linear-time input family;
regularity alone does not bound $K$.
The nodal estimate does not extend to linear interpolation, whose
chartwise error is at most $QC_0h^4+Ch^2/8$.
Adaptive accepted-step costs follow from
Proposition~\ref{prop:adaptive-geometry}, with $\mathcal W=O(N)$.

For bounded input points, normalized Gaussian sums have derivative
bounds independent of $N$: the first derivatives are $O(\sigma^{-1})$
and the second derivatives are $O(\sigma^{-2}+\sigma^{-1})$ in fixed
unit-normal charts. This follows from boundedness of $\phi,\phi'$
and the chain rule. With the other path constants bounded,
\eqref{eq:condition-count} gives $\Gamma=O(1+\sigma^{-2})$.
This upper derivative estimate does not control $\alpha^{-1}$.
Even the two-point set $\{-1,1\}$ has median derivative
$\phi(1/\sigma)/\sigma$, exponentially small for small bandwidth.

\paragraph{Final accuracy and exact reconstruction.}
If only the endpoint needs higher accuracy, any admissible step
can be kept fixed and additional Newton iterations used at $t=1$.
By~\eqref{eq:main-count}, residual $\delta$ requires at most
$m=O(\log(1+\log^+(C_1/\delta)))$ additional iterations.
The total work is then
\begin{equation}\label{eq:atomic-work}
 O\bigl(N(K+m+1)\bigr),
\end{equation}
and the intrinsic endpoint error is at most $2Q\delta/\alpha$.

For an exact discrete cut, use the residual
$1/(8n_{\max})$ and the bandwidth-halving procedure above.
If $K_j,m_j$ are the continuation and refinement counts at scale
$j$, its total arithmetic work is
\begin{equation}\label{eq:multiscale-work}
 O\left(N\sum_{j=0}^{J_{\rm scale}}(K_j+m_j+1)\right).
\end{equation}
Equation~\eqref{eq:bandwidth-reductions} bounds the reductions once
the requested residuals are attained. Each pass remains linear in $N$;
an adaptive implementation also counts rejected steps and extra corrections.

Inexact evaluation is covered in Section~\ref{sec:numerical}; see
\cite{BatesPrecision} for adaptive precision. Our unit-cost Gaussian
evaluations specify arithmetic rather than bit complexity.

\begin{remark}[Weights]
Positive normalized weights change neither the finite-sum cost nor
the discrete-limit proof. Rational weights permit exact integer
half-space tests after clearing denominators. The reconstruction
theorems above concern equally weighted odd sets in general position.
\end{remark}

\section{Other spherical and fixed-point problems}\label{sec:other}
The projection model applies beyond hyperplane partitions.
Here we specify the test maps and the meaning of their residuals;
the underlying existence results are classical~\cite{Matousek,DeLongueville}.
The numerical bounds require the regularity hypotheses of Section~\ref{sec:numerical}.

\subsection{Combinatorial and geometric encodings}
\paragraph{Lusternik--Schnirelmann coverings.}
If $S^n=A_1\cup\cdots\cup A_{n+1}$ with nonempty closed $A_i$,
the covering theorem asserts that some $A_i$ contains an antipodal
pair. Its standard distance-map proof~\cite[Section~2.1]{Matousek}
gives the odd test map
\[
 f_i(x)=\dist(x,A_i)-\dist(-x,A_i),\qquad 1\le i\le n.
\]
At a zero, membership of either point in any $A_i$, $i\le n$,
implies membership of both; otherwise both belong to $A_{n+1}$.
The same argument shows that $\|f(x)\|_\infty\le\varepsilon$ yields
an index $i$ with
$\max\{\dist(x,A_i),\dist(-x,A_i)\}\le\varepsilon$.
Distance functions are only Lipschitz in general. Approximating each
by a smooth function within $\delta$ gives an odd smooth test map;
residual $\varepsilon$ then gives distance bound $\varepsilon+2\delta$,
not necessarily exact membership in the original closed sets.

For an open cover $U_1,\ldots,U_{n+1}$, choose a smooth subordinate
partition of unity $\phi=(\phi_1,\ldots,\phi_{n+1})$.
Let $L:\{v:\sum_i v_i=0\}\to\R^n$ be a linear isometry and set
$f(x)=L(\phi(x)-\phi(-x))$.
This smooth odd map has a quantitative decoding rule: if
$\|f(x)\|_2<1/(n+1)$, choose $i$ maximizing $\phi_i(x)$.
Then $\phi_i(x)\ge1/(n+1)>|\phi_i(x)-\phi_i(-x)|$, so both
values are positive and $x,-x\in U_i$.
Evaluation costs include these distance or partition-of-unity oracles.

\paragraph{Necklace bisection.}
An $n$-colored necklace with even color counts admits an equal
division between two recipients with at most $n$ cuts.
The continuous version bisects $n$ finite nonatomic measures $\mu_i$
on $[0,1]$ using signed intervals~\cite{AlonWest}.
For $x\in S^n$, put
\[
 \ell_j=\frac{x_j^4}{\sum_{k=1}^{n+1}x_k^4},\quad
 s_j=\sum_{k=1}^j\ell_k,\quad s_0=0,\qquad
 f_i(x)=\sum_{j=1}^{n+1}\sgn(x_j)\mu_i([s_{j-1},s_j]).
\]
Collapsed intervals contribute zero. Thus $f:S^n\to\R^n$ is
continuous and odd; its zeros give the desired bisections, and
$|f_i(x)|$ is exactly the imbalance for color $i$.
The fourth powers replace the usual squares to improve smoothness.
For $C^2$ densities $\rho_i$, each signed term equals
\[
 \frac{x_j|x_j|^3}{\sum_kx_k^4}
 \int_0^1\rho_i(s_{j-1}+u\ell_j)\,du,
\]
which is $C^2$, including at $x_j=0$. Hence the same projection
model and conditional numerical estimates apply.
For actual beads, the classical interval encoding requires a separate
rounding step~\cite[Section~3.2]{Matousek}; small smooth residual
alone does not certify an exact discrete split.

\paragraph{Topological Radon points.}
Let $a:\Delta^{d+1}\to\R^d$ be continuous, and put
$H=\{u\in\R^{d+2}:\sum_i u_i=0\}$.
On its unit sphere $S(H)\cong S^d$, define
\[
 P_i(u)=\frac{(u_i)_+^3}{\sum_j(u_j)_+^3},\qquad
 f_a(u)=a(P(u))-a(P(-u)),\qquad v_+=\max\{v,0\}.
\]
Both denominators are positive: a nonzero vector in $H$ has both
positive and negative coordinates. The barycentric points $P(u)$
and $P(-u)$ have disjoint supports, and $f_a$ is odd. Consequently
a zero yields points of disjoint faces with the same image, proving
the topological Radon conclusion. An approximate zero bounds the
distance between their images by the same residual.
If $a$ extends to a $C^2$ map near the simplex, $f_a$ is $C^2$,
since $v_+^3$ is $C^2$; the spherical numerical theorems apply under
their rank hypotheses. Cubing replaces the nonsmooth positive-part
parametrization and does not change the existence argument.

These examples are constructive zero problems. Other uses of
Borsuk--Ulam, such as ruling out a coloring, need a separate
decoding argument before they become numerical search problems.

\subsection{Fixed points and equilibrium problems}\label{sec:fixedpoints}
A fixed-point problem can also be encoded on a sphere. Su~\cite{Su}
gave a direct constructive reduction of Brouwer to Borsuk--Ulam.
For numerical purposes the following smooth version additionally
preserves the regularity of every fixed point.

\begin{prop}[A smooth spherical encoding]\label{prop:brouwer-encoding}
Let $T:\overline B^n\to B^n$ be $C^2$, with a $C^2$ extension to a
neighborhood of the closed ball. There is an odd $C^2$ map
$f_T:S^n\to\R^n$ whose zero pairs correspond bijectively to the fixed
points of $T$. A pair is regular precisely when $I-DT$ is invertible
at the corresponding fixed point.
\end{prop}
\begin{proof}
Compactness gives $r<1$ bounding the norm of an extension on a slightly
larger ball. Multiply this extension by a smooth cutoff, equal to one
on $\overline B^n$ and supported in that larger ball. The resulting
$\widetilde T:\R^n\to\R^n$ is $C^2$, compactly supported, and satisfies
$\|\widetilde T\|\le r$. On $S^n\subset\R^n\times\R$ set
\begin{equation}\label{eq:brouwer-sphere}
 f_T(u,s)=
 \begin{cases}u-s\widetilde T(u/s),&s\ne0,\\u,&s=0.\end{cases}
\end{equation}
It is odd. Near the equator $\widetilde T(u/s)=0$, so it is $C^2$
there as well. An equatorial point is not a zero. At any other zero,
$y=u/s=\widetilde T(y)$ has norm at most $r$, hence $T(y)=y$.
Conversely each fixed point gives the pair
$\pm\Xi(y)$, where $\Xi(y)=(y,1)/\sqrt{1+\|y\|^2}$.
In this chart
\[
 f_T(\Xi(y))=\frac{y-\widetilde T(y)}{\sqrt{1+\|y\|^2}}.
\]
At a zero its derivative is $(I-DT(y))/\sqrt{1+\|y\|^2}$,
proving the regularity assertion.
\end{proof}

This encoding also controls approximate answers. If
$\|f_T(u,s)\|\le\varepsilon\le1/4$, then $|s|>1/2$: otherwise
$\|f_T\|\ge\sqrt{1-s^2}-r|s|>1/4$.
If also $2\varepsilon<1-r$, then $y=u/s$ lies in $B^n$ and
$\|y-T(y)\|\le2\varepsilon$. Indeed,
$\|y-\widetilde T(y)\|=\|f_T(u,s)\|/|s|\le2\varepsilon$ and
$\|y\|\le r+2\varepsilon<1$.
For continuous self-maps of the closed ball, replacing $T$ by
$(1-\delta)T$ and uniformly smoothing gives interior maps. Approximate
fixed points then have the additional error $\delta$ plus the smoothing
error; preservation of every degenerate fixed point is not asserted.

Nash's continuous map on the product of strategy simplices has exactly
the equilibria as fixed points~\cite{Nash}. A ball parametrization gives
a continuous spherical encoding; computational reductions to Borsuk--Ulam
are also known~\cite{DeligkasEtAl}. Positive-part formulas need not be
differentiable at equilibria, so our smooth tracking results require a
suitable formulation and an approximate-zero-to-regret bound.
Numerical equilibrium continuation includes
\cite{HeringsPeeters2001,GovindanWilson,HeringsPeeters2010}.
For all equilibria of generic games, Herings--Peeters~\cite{HeringsPeeters2005}
enumerate supports, solve their polynomial systems by homotopy, and test
the inequalities. These are distinct search strategies, with different
completeness conditions from Section~\ref{sec:all}.

\subsection{Periodic solutions and other finite-dimensional problems}
A periodic trajectory can be found by searching for its initial state.
Suppose $\dot x=v(\tau,x)$ has period $\tau_0$ in the time variable,
unique solutions on $[0,\tau_0]$, and a continuous time-$\tau_0$ map
$P:K\to K$ on a nonempty compact convex set. Fixed points of $P$ give periodic
solutions, the classical Brouwer application discussed in~\cite{CidMawhin}.
For $K$ a ball, a $C^2$ map $P$ sending it strictly into its interior
satisfies Proposition~\ref{prop:brouwer-encoding}. Regularity is exactly
$\det(I-DP)\ne0$, whereas stability requires stronger multiplier bounds.
Thus the model search can also reach unstable periodic solutions.
The derivative is obtained by integrating
$\dot Y=D_xv(\tau,x(\tau))Y$, $Y(0)=I$, alongside the trajectory;
the endpoint mismatch $\|P(x)-x\|$ is a directly interpretable residual.
For autonomous nonconstant periodic orbits the phase direction gives a
multiplier $1$; a transverse section or phase condition is then needed.

A $C^1$ Galerkin system $A_m(a)=0$ with $\langle A_m(a),a\rangle>0$
on $\|a\|=R$ similarly admits the model $g(a)=a$: their linear homotopy
has positive radial inner product on the boundary, so
Proposition~\ref{prop:general-cylinder} applies under its rank hypotheses.
The limit $m\to\infty$ needs separate uniform estimates. Variational
inequalities on a nonempty compact convex $K$ give another Brouwer map:
$\langle A(x),y-x\rangle\ge0$ for all $y\in K$ is equivalent to
$x=\Pi_K(x-\lambda A(x))$, $\lambda>0$, by the projection property.
For polytopes, projection is piecewise affine and active-face transitions
require attention. Economic fixed-point algorithms~\cite{Scarf} provide
further applications; existence alone does not imply convergence of the
underlying dynamics.

\section{Numerical experiments}\label{sec:experiments}
The tests in this section use tangent prediction and Newton correction.
The accompanying Python scripts and recorded outputs are provided as
ancillary files; they specify all inputs for these tests, and no external
data are needed.

\paragraph{Accuracy and a narrow difficult part of the path.}
On $S^1$, parametrized by $\theta$, consider
\[
 F_\delta(\theta,t)=(1+0.3t\sin2\theta)
       ((1-2t)\sin\theta+\delta t\cos\theta),\quad\delta>0.
\]
This is odd under $\theta\mapsto\theta+\pi$, starts at the projection
$\sin\theta$, and has the exact lifted branch
$\theta(t)=-\operatorname{atan2}(\delta t,1-2t)$.
Writing $q=(1-2t)^2+\delta^2t^2$, one has
$|\theta'|=\delta/q$ and $F_\theta=(1+0.3t\sin2\theta)\sqrt q$
on the branch. Hence near $t=1/2$ the speed is of order
$\delta^{-1}$ and the spatial inverse norm is of order
$\delta^{-1}$, although the total angular variation stays below $\pi$.

For $\delta=0.4$, uniform meshes with $K=128,256,512,1024$
give maximum nodal errors $1.062\cdot10^{-6}$,
$6.856\cdot10^{-8}$, $4.355\cdot10^{-9}$, and
$2.741\cdot10^{-10}$, respectively. These runs use 60-digit
arithmetic. The successive observed orders are $3.953$, $3.977$,
and $3.990$, consistent with the fourth-order nodal estimate.

An adaptive controller rejects a trial
if the predicted angular displacement exceeds $0.1$, the corrected
Newton residual ratio $|F/F_\theta|$ exceeds $10^{-6}$, or
$F_\theta\le0$. It starts with $h=0.05$, halves rejected steps,
and enlarges accepted steps by $1.4$, capped at $0.1$.
The exact branch is used only to measure error. This controller
does not validate whole tubes as required by
Proposition~\ref{prop:adaptive-geometry}.
\begin{center}
\begin{tabular}{r|r|rr|r}
$\delta$ & uniform $K$ & accepted & rejected & adaptive max. error\\\hline
$1$ & $128$ & $52$ & $24$ & $9.21\cdot10^{-7}$\\
$0.1$ & $512$ & $101$ & $48$ & $9.70\cdot10^{-7}$\\
$0.01$ & $8192$ & $126$ & $58$ & $9.72\cdot10^{-7}$
\end{tabular}
\end{center}
The uniform count is the first tested power of two, starting at 16,
with measured error at most $10^{-6}$. Every trial costs two scalar
solves. These observations do not prove a general complexity bound.

\paragraph{Point-sum cost.}
For two deterministic random point clouds (seed 20261103), evaluating
the normalized Gaussian values and both planar derivative rows took
$0.044$, $0.354$, and $4.402$ milliseconds for total sizes
$10^3,10^4,10^5$. These are medians of seven batches of twenty
evaluations, in Python 3.12/NumPy 2.3.5 on an x86-64 host, with
one BLAS thread. This checks evaluation scaling only, not the number
of continuation steps; the latter depends on the selected path.
All experiments except the order test use double precision.

\paragraph{Distinct periodic solutions, including an unstable one.}
Consider $\dot x=x-x^3+0.1\sin(2\pi\tau)$ and its time-one map $P$.
The interval $[-2,2]$ is strictly inward-pointing. Scale $x=2y$ and use
$\widetilde T(y)=\chi(y)P(2y)/2$, with the $C^2$ quintic cutoff specified
in \texttt{periodic.py}: it equals one for $|y|\le1$ and zero for
$|y|\ge2$. This extension has norm less than one, since initial states
in $[-4,4]$ enter $(-2,2)$ in time at most $2/5.9<1$.

On $S^1$ use $(u,s)=(\cos\theta,\sin\theta)$ and the models
$g_p(\theta)=f_T'(p)\sin(\theta-p)$ from~\eqref{eq:pair-model}.
Sixteen starts $p=(j+1/2)\pi/16$, $0\le j<16$, all reached the target,
giving the following initial states and multipliers:
\begin{center}
\begin{tabular}{r|r|r}
initial state $x$ & $P'(x)$ & starts reaching it\\\hline
$-1.014260752$ & $0.135428702$ & $7$\\
$-0.015522478$ & $2.717274698$ & $2$\\
$\phantom{-}0.985356306$ & $0.135428702$ & $7$
\end{tabular}
\end{center}
The middle solution is unstable, but regular. Adaptive tangent prediction
with up to five Newton corrections took 26--33 accepted steps.
DOP853 tolerances were $(2\cdot10^{-10},2\cdot10^{-12})$; independent
re-evaluation at $(2\cdot10^{-13},2\cdot10^{-15})$ gave maximum
fixed-point residual $1.6\cdot10^{-12}$. The script records all settings.
This demonstrates distinct reachable endpoints, including an unstable one;
it is not an interval-validated exhaustion of the hemisphere.

\bigskip
\noindent\textbf{Acknowledgments.} The author thanks Roman Karasev and
Alexey Fakhrutdinov for helpful comments and discussions.

\bigskip
\noindent\textbf{AI use statement.} LLMs were used for proof refinement,
literature review, and manuscript revision.

\bigskip
\noindent\textbf{Oleg R. Musin}\\
The University of Texas Rio Grande Valley, Brownsville, Texas, USA\\
\href{mailto:oleg.musin@utrgv.edu}{\texttt{oleg.musin@utrgv.edu}}
\end{document}